\documentclass[a4paper,11pt, oneside]{amsart}
\title{Principal Pairs of Quantum Homogeneous Spaces}
\usepackage[english]{babel}
\usepackage{amsthm}
\usepackage{amssymb}
\usepackage{amsfonts}
\usepackage{amsmath}
\usepackage{amsthm}
\usepackage[all]{xy}
\usepackage{geometry}
\usepackage{ae,aecompl}
\usepackage{eurosym}
\usepackage{verbatim}
\usepackage{mathrsfs}
\usepackage{caption}
\usepackage{url}
\usepackage{tikz}
\usepackage{hyperref}

\newcounter{stepcounter}
\newtheoremstyle{smallcaps}
 {3pt} 
 {3pt} 
 {\itshape} 
 {} 
 {\sc} 
 {.} 
 {.5em} 
 {} 
 
\newtheoremstyle{smallcapsdef}
 {3pt} 
 {3pt} 
 {} 
 {} 
 {\sc} 
 {.} 
 {.5em} 
 {} 

\theoremstyle{plain}

\newtheorem{thm}{Theorem}[section]

\newtheorem{lem}[thm]{Lemma}
\newtheorem{prop}[thm]{Proposition}
\newtheorem{cor}[thm]{Corollary}

\theoremstyle{definition}

\newtheorem{eg}[thm]{Example}
\newtheorem{defn}[thm]{Definition}

\newtheorem{remark}[thm]{Remark}

\date{}

\newcommand\bit{\begin{itemize}}
\newcommand\eit{\end{itemize}}
\newcommand\bet{\begin{enumerate}}
\newcommand\eet{\end{enumerate}}
\newcommand\ed{\end{document}}

\DeclareFontFamily{U}{mathx}{\hyphenchar\font45}
\DeclareFontShape{U}{mathx}{m}{n}{
 <5> <6> <7> <8> <9> <10>
 <10.95> <12> <14.4> <17.28> <20.74> <24.88>
 mathx10
 }{}
\DeclareSymbolFont{mathx}{U}{mathx}{m}{n}
\DeclareFontSubstitution{U}{mathx}{m}{n}
\DeclareMathAccent{\widecheck}{0}{mathx}{"71}
\DeclareMathAccent{\wideparen}{0}{mathx}{"75}

\renewcommand{\a}{\alpha}
\renewcommand{\d}{\delta}
\newcommand{\e}{\varepsilon}

\newcommand\DEL{\Delta}

\newcommand\bC{{\mathbb C}}

\newcommand\bR{{\mathbb R}}

\newcommand\F{{\mathcal F}}

\newcommand\N{{\mathcal N}}
\renewcommand{\O}{\mathcal{O}}

\newcommand\can{\mathrm{can}}

\newcommand\co{\mathrm{co}}

\newcommand\unit{\mathrm{U}}
\newcommand\counit{\mathrm{C}}

\newcommand\id{\mathrm{id}}
\newcommand\proj{\mathrm{proj}}

\newcommand\inv{^{-1}}

\newcommand\oby{\otimes}

\newcommand\sseq{\subseteq}

\def\qbinom#1#2{\ensuremath{\left[\kern-.3em\left[\genfrac{}{}{0pt}{}{#1}{#2}\right]\kern-.3em\right]_q}}

\newcommand\mto{\mapsto}

\newcommand\algn{algebra}

\newcommand\EE{{\mathcal E}}

\newcommand{\OO}{\mathcal{O}}
\newcommand{\qphs}{\OO_q(G/L^{\,\mathrm{s}}_S)}

\newcommand{\uslevi}{U_q(\frak{l}_S)}
\newcommand{\sslevi}{\OO_q(L^{\,\mathrm{s}}_S)}
\newcommand{\usslevi}{U_q(\frak{l}^{\,\mathrm{s}}_S)}

\usepackage{tikz}
\usetikzlibrary{decorations.pathreplacing}

\usepackage{ytableau}

\author[A. Carotenuto]{Alessandro Carotenuto}
\address{Institute of Mathematics, Czech Academy of Sciences, \v{Z}itn\'a 25, 115 67 Prague, Czech Republic}
\email{carotenuto@math.cas.cz, acaroten91@gmail.com}

\author[R. \'O Buachalla]{R\'eamonn \'O Buachalla}
\address{Mathematical Institute of Charles University, Sokolovsk\'a 83, Prague, Czech Republic} \email{obuachalla@karlin.mff.cuni.cz}

\usepackage[english]{babel}

\thanks{AC is supported by the GA\v{C}R project 20-17488Y and \mbox{RVO: 67985840}. R\'OB was supported by GA\v{C}R through the framework of the grant GA19-06357S, and is supported by the Charles University PRIMUS grant \emph{Spectral Noncommutative Geometry of Quantum Flag Manifolds} PRIMUS/21/SCI/026}

\begin{document}

\begin{abstract}
We propose a simple but effective framework for producing examples
of covariant faithfully flat (generalised) Hopf--Galois extensions from a nested pair of quantum homogeneous spaces. Our construction is modelled on the classical situation of a homogeneous fibration $G/N \to G/M$, for $G$ a group, and $N \subseteq M \subseteq G$ subgroups. Variations on Takeuchi's equivalence and Schneider's descent theorem are presented in this context. Quantum flag manifolds and their associated quantum Poisson homogeneous spaces are taken as motivating examples. Moreover, a large collection of \emph{noncommutative fibrations} (in the spirit of Brzezi\'{n}ski and Szyma\'{n}ski) are constructed. 
\end{abstract}

\maketitle 

\tableofcontents

\section{Introduction}

From their purely algebraic origins \cite{ChSw69,KrTa81}, Hopf--Galois extensions have become a central structure in noncommutative geometry where they are rightly considered as noncommutative analogues of principal bundles \cite{BeggsMajid:Leabh}. 
More recently, efforts have begun to extend this work to a theory of noncommutative fiber bundles. This has focused on the study of noncommutative sphere bundles from both a Hopf algebraic \cite{TBWS} and a $C^*$-algebraic point of view \cite{ArKa20}. In this paper we propose a simple but effective new framework for producing examples of noncommutative fibrations, both principal and non-principal, from a nested pair of quantum homogeneous spaces. Our construction generalises the prototypical situation of a homogeneous fibration, probably the simplest type of fibration after a homogeneous space: let $G$ be a group, then we have a fibration
\begin{align*}
G/N \twoheadrightarrow G/M, & & \textrm{ for any two subgroups } N \subseteq M \subseteq G.
\end{align*}
This fibration is principal if and only if $N$ is a normal subgroup of $M$, see for example \cite{OVLie}. We observe that this construction extends directly to the Hopf algebra setting, providing a large and natural family of principal comodule algebras. In more detail, we start with a nested pair of quantum homogeneous spaces $B \subseteq P$ (coideal subalgebras satisfying a faithful flatness condition) and show that $P$ is a principal $\pi_B(P)$-comodule algebra with base $B$, where $\pi_B(P)$ is a Hopf algebra of coinvariants playing the role of the quotient group $M/N$. Moreover, we have direct generalisations of Takeuchi's and Schneider's categorical equivalences to this setting. We call such a pair of quantum homogeneous spaces $(B,P)$ a \emph{principal pair}.


The space of coinvariants $\pi_B(P)$ is of course not guaranteed to have a coalgebra structure, just as $N$ is not guaranteed to be normal in $M$. However, even then the pair $B \subseteq P$ retains many attractive algebraic properties: $P$ is a faithfully flat extension of $B$ and the extension satisfies a generalised Hopf--Galois condition. The generalised version of Takeuchi's equivalence still holds and $P$ can be described as an $(A,B)$-relative module associated to the fiber $\pi_B(P)$. This list of properties places these examples firmly within the domain of Brzezi\'{n}ski and Szyma\'{n}ski's putative theory of \emph{noncommutative fiber bundles}, providing a rich family of examples deserving to be considered as \emph{noncommutative fibrations}.


Of particular concern to this paper is the construction of principle comodule $\OO(\mathbb{T}^k)$-algebras, where $\mathbb{T}^k$ is the $k$-torus. We are motivated by the classical situation of a principal torus fibration $G/L^{\mathrm{s}} \to G/L$, where $G$ is a compact Lie group, $L$ is a reductive subgroup, and $L^{\mathrm{s}}$ is the semisimple part of $L$. We start with a quantum homogeneous space $B = A^{\co(H)}$, where we assume that $H$ is cosemisimple and all of its grouplike elements are central. We then consider $K_H$ the Hopf algebra quotient of $H$ by the augmentation ideal of $\Gamma_H$, the group Hopf algebra of the group of grouplike elements of $H$. The pair of quantum homogeneous spaces $B$ and $P$ associated to the projections $\pi_H:A \to H$ and $\pi_A: A \to K_H$ is then shown to be a principal pair, which we call a \emph{group algebra principal pair}. We also give a formal treatment of relative line modules in terms of an algebra grading over the group of grouplike elements of $H$. 


In the Hopf algebraic approach to noncommutative geometry, Majid and Brzezi\'{n}ski's quantum principal bundles and principal connections play a central role, see for example the recent monograph \cite{BeggsMajid:Leabh}. An alternative definition of a principal pair was introduced by Mrozinski and the authors in \cite{BWGrass} as a formal framework in which to construct quantum principal bundles and quantum principal connections over the quantum Grassmannians. The fact that any nested pair of quantum homogeneous spaces gives a principal comodule algebra easily implies the equivalence of the two definitions. This gives support to the framework introduced in \cite{BWGrass} while offering a significant simplification of the setup, providing greater clarify, and elucidating the relationship with the classical situation. 
%
%
%


One of our main motivations for introducing group algebra principal pairs is to provide a firmly Hopf algebraic framework for studying relative line modules and principal connections over quantum flag manifolds $\OO_q(G/L_S)$. 
Taking the quantum Poisson homogeneous space $\OO_q(G/L^{\mathrm{s}}_S)$ as the total space, the homogeneous spaces $\OO_q(G/L^{\mathrm{s}}_S)$ and $\OO_q(G/L_S)$ are easily shown to give a group algebra principal pair. This connects the quantum flag manifolds to the powerful machinery of principal comodule algebras and quantum principal bundles. Concretely, we get a graded algebra structure on the quantum Poisson homogeneous space whose homogeneous summands are relative line modules, and every relative line module is of this form. Moreover, we get a very useful description of the line modules in terms of the standard generators of $\OO_q(G/L_S^{\mathrm{s}})$ (as first considered by Stokman \cite{Stok}). 


The principal pair presentation of the quantum flag manifolds leads to a number of important applications in accompanying papers. The first is the description of the Heckenberger--Kolb calculi \cite{HK,HKdR} in quantum principal bundle terms \cite{CDOBBW}, generalising the case of quantum projective space given in \cite{MMF1}, and the more general quantum Grassmannian case given in \cite{BWGrass}. As for the quantum Grassmannian case, this provides a framework in which the Borel--Weil theorem for the irreducible flag manifolds can be $q$-deformed \cite{CDOBBW}. The Borel--Weil theorem has itself a number of important consequences. It allows us to identify which line modules over the irreducible quantum flag manifolds are positive and which are negative \cite{DOKSS}. In turn, in \cite{OSV} the negative line modules are used to twist the Dolbeault--Dirac operators of the irreducible quantum flag manifolds so as to produce Fredholm operators. In forthcoming work, the principal comodule algebra presentation of $\OO_q(G/L_S^{\mathrm{s}})$ will be used to present the $C^*$-completion of $\OO_q(G/L_S^{\mathrm{s}})$ as a Cuntz--Pimsner algebra, generalising the case of quantum projective space treated in \cite{ArBrLa15}. 

\subsection{Summary of the Paper}

The paper is organised as follows. In Section 2 we recall the standard definitions and results about principal comodule algebras, Schneider's descent theorem, and Takeuchi's equivalence for relative Hopf modules.

In Section 3, we prove a faithful flatness condition, as well as a generalised Hopf--Galois condition, for a nested pair of quantum homogeneous spaces. Building on this, the notion of a principal is then introduced. Principal pairs are shown to give principal comodule algebras and hence we regain the earlier formulation of a principal pair introduced by the authors and Mrozinski. Along the way we establish a generalisation Takeuchi's equivalence for nested pairs of quantum homogeneous spaces,  as well as a refinement of Schneider's equivalence taking into account the extra symmetry given by a principal pair.

In Section 4 we introduce the definition of a group algebra principal pair. More explicitly, we prove the following theorem. 
\begin{thm}
Let $A$ and $H$ be Hopf algebras, $\pi_H:A \to H$ a surjective Hopf algebra map, and $\proj_K:H \to K_H$ the canonical projection. If $A$ is cosemisimple and all the grouplike elements of $H$ are central in $H$, then the pair
$$
\left(\pi_H:A \to H, \, \proj_K \circ \pi_{H}: A \to K_{H}\right)
$$
is a principal pair. 
\end{thm}
We then show that we have an associated decomposition of the total space $P = A^{\co(K_H)}$ into a graded algebra whose summands are relative line modules, and show that every relative line module over the base space is of this form.

In Section 5 we treat our motivating family of examples, the quantum flag manifolds. Beginning with the necessary definitions of Drinfeld--Jimbo quantised enveloping algebras, quantum coordinate algebras, and the quantum flag manifolds, we then verify the group algebra principal pair condition by establishing the following theorem.

\begin{thm}
For any quantum flag manifold $\OO_q(G/L_S)$, the pair
$$
\Big(\OO_q(G/L_S), \, \OO_q(G/L^{\mathrm{s}}_S)\Big)
$$
is a group algebra principal pair.
\end{thm}

This gives a concise description of the relative line modules in terms of a $\mathcal{P}_{S^c}$-algebra grading on $\OO_q(G/L^{\mathrm{s}}_S)$. We then use this description to construct an explicit left $A$-covariant principal connection for the principal comodule algebra, generalising the well-known $q$-monopole connection for the special case of the Podle\'s sphere. Non-cleftness of the Hopf--Galois extension is also established.

In Section 6, we produce two families of \emph{noncommutative fibrations} in the spirit of Brzezi\'{n}ski and Szyma\'{n}ski, each associated to a nested pair of subsets of the simple roots of a complex semisimple Lie algebra $\mathfrak{g}$. Examples include noncommutative fibrations with quantum sphere, quantum Stiefel manifold, and quantum flag manifold fibers.

We finish with an appendix discussing some relevant details of quantum homogeneous spaces.

\subsection*{Acknowledgments} We would like to thank Andrey Krutov, Karen Strung, Hans--Jurgen Schneider, Paolo Saracco, Rita Fioresi, Marco Matassa, Tomasz Brzezi\'{n}ski, and Wojciech Szyma\'{n}ski for helpful discussions.

\section{Preliminaries on Principal Comodule Algebras}

In this section we the necessary results on principal comodule algebras, Schneider's equivalence, relative Hopf modules, and Takeuchi's equivalence. All this material is by now quite well known, and a more detailed presentation can be found in the monographs \cite{BeggsMajid:Leabh,TBGS}.

Throughout this section, and indeed the paper, all algebras will be unital and defined over $\mathbb{C}$, and all Hopf algebras will be assumed to have a bijective antipode. We denote the coproduct, counit, and antipode of a Hopf algebra $H$ by $\Delta, \epsilon$, and $S$ respectively, and we denote the cotensor product over $H$ by $\square_H$. 

\subsection{Principal Comodule Algebras} \label{subsection:PCA}

A right $H$-comodule algebra $(P,\Delta_R)$ is said to be a {\em $H$-Hopf--Galois extension of} $B := P^{\co(H)}$ 
if for $m_P:P \otimes_B P \to P$ the multiplication of~$P$,
a bijection is given by 
$$
\can := (m_P \otimes \id) \circ (\id \otimes \DEL_R): P \otimes_B P \to P \otimes H.
$$
We call $P$ the \emph{total algebra} and $B$ the \emph{base algebra}. We say that $P$ is faithfully flat as a right $B$-module if the functor 
$$
P \otimes_B -: {}_B\mathrm{Mod} \to \mathrm{Mod_{\mathbb{C}}}
$$
preserves and reflects exact sequences. Faithful flatness as a left $B$-module is defined analogously.

\begin{defn}
A {\em principal right $H$-comodule algebra} is a right $H$-comodule algebra $(P,\DEL_R)$ such that $P$ is a Hopf--Galois extension of $B := P^{\co(H)}$ and $P$ is faithfully flat as a right and left $B$-module.
\end{defn}

As shown in \cite{BrzBohm,BRzHajComptesT}, a right $H$-comodule algebra $(P,\Delta_R)$ is principal if and only if there exists a principal $\ell$-map $\ell: H \to P \oby P$, as defined below.

\begin{defn}\label{defn:ell}
For a right $H$-comodule algebra $(P,\Delta_R)$, a {\em principal $\ell$-map} is a linear map $\ell:H \to P \otimes P$ satisfying 
\bet
\item $\ell(1_H) = 1_P \oby 1_P$,
\item $m_P \circ \ell = \e_H 1_P$,
\item $(\ell \oby \id_H) \circ \DEL = (\id_P \oby \DEL_R) \circ \ell$,
\item $(\id_H \oby \ell) \circ \DEL = (\DEL_L \oby \id_P) \circ \ell$,
\eet
where, for $\text{flip}: P \oby H \to H \oby P$ the flip map, $\DEL_L : = (S \oby \id) \circ \text{flip} \circ \Delta_R$.
\end{defn}

\subsection{Schneider's Descent Theorem}

Let $H$ be a Hopf algebra, $(P,\Delta_R)$ a right $H$-comodule algebra, and denote $B := P^{\co(H)}$. Then $(P,\Delta_R)$ is a principal comodule algebra if and only if the functor
$$
P \otimes_B -: {}_B \mathrm{Mod} \to {}_P\mathrm{Mod}^H
$$
induces an equivalence of categories. In this case the inverse functor to $P \otimes_B -$ is 
\begin{align*}
{}_P\mathrm{Mod}^H \to {}_B\mathrm{Mod}, & & \mathcal{F} \mapsto \mathcal{F}^{\co(H)}.
\end{align*}
This result is known as \emph{Schneider's descent theorem}, but we also find it convenient to refer to it as \emph{Schneider's equivalence}, see \cite[Theorem I]{Schneider90} for a proof. Explicitly, the unit natural isomorphism for the equivalence is given by
\begin{align} \label{eqn:Sunit}
\unit: \mathcal{M} \to (P\otimes_B \mathcal{M})^{\co(\pi_B(P))}, & & m \mto 1 \otimes m,
\end{align}
while the counit isomorphism is given by
\begin{align} \label{eqn:Scounit}
\qquad \counit:P\otimes_B \mathcal{N}^{\co(\pi_B(P))}\to \mathcal{N}, \, & & p \otimes n \mto pn.
\end{align}


\subsection{Relative Hopf Modules} 

We say that a left coideal subalgebra $B \sseq A$ is a \emph{quantum homogeneous $A$-space} if $A$ is faithfully flat as a right $B$-module and $B^+A = AB^+$. (See Appendix \ref{app:QHSIdealsQSGs} for some equivalent presentations of the definition.) 
It follows from \cite[Theorem 1]{Tak} that, for the Hopf algebra surjection $\pi_B:A \to A/B^+A$, and the associated right $\pi_B(A)$-coaction $\Delta_{R,\pi_B} := (\id \otimes \pi_B) \circ \Delta$, the space of coinvariants $A^{\co(A/B^+A)}$ is equal to $B$.

We denote by~${}^A_B\mathrm{Mod}$ the category of \emph{relative Hopf modules}, that is, the category whose objects are left \mbox{$A$-comodules} \mbox{$\DEL_L:\mathcal{F} \to A \otimes \mathcal{F}$}, endowed with a left $B$-module structure such that, for all $f \in \mathcal{F},$ \mbox{$b \in B$}, we have $\DEL_L(bf) = \Delta_L(b)\DEL_L(f)$, 
and whose morphisms are left $A$-comodule, left $B$-module, maps. As explained in Appendix \ref{subsection:PrinConns}, very relative Hopf module is projective as a left $B$-module. 
 A \emph{relative line module} over $B$ is an invertible object $\EE$ in the cateogry ${}^A_B\mathrm{Mod}_0$. 

\subsection{Takeuchi's Equivalence}
In the following we denote by ${}^{\pi_B}\mathrm{Mod}$ the category whose objects are left \mbox{$\pi_B(A)$-comodules}, and whose morphisms are left $\pi_B(A)$-comodule maps. We use similar notation for right $\pi_B(A)$-comodules. 

Define a functor $\Phi:{}^A_B\mathrm{Mod} \to {}^{\pi_B}\mathrm{Mod}$ by setting $\Phi(\mathcal{F}) := \mathcal{F}/B^+\mathcal{F}$, for $B^+ := B \cap \ker(\e)$, where the left $\pi_B(A)$-comodule structure of $\Phi(\mathcal{\F})$ is given by 
$
\Delta_L[f] := \pi_B(f_{(-1)})\otimes [f_{(0)}],
$
with square brackets denoting the coset of an element in $\Phi(\mathcal{\F})$. In the other direction, we use the cotensor product $\square_{\pi_B(A)}$, which we find convenient to denote by $\square_{\pi_B}$. Define a functor $\Psi: {}^{\pi_B}\mathrm{Mod} \to {}^A_B\mathrm{Mod}$ by setting $\Psi(V) := A \,\square_{\pi_B} V$, where the left $B$-module and left $A$-comodule structures of $\Psi(V)$ are defined on the first tensor factor, and if $\gamma$ is a morphism in ${}^{\pi_B}\mathrm{Mod}$, then $\Psi(\gamma) := \id \otimes \gamma$. Note that $A$ is naturally an object in ${}^A_B\mathrm{Mod}$, and that $\Phi(A) = \pi_P(A).$

As established in~\cite[Theorem 1]{Tak}, an adjoint equivalence of categories between~${}^A_B\mathrm{Mod}$ and~${}^{\pi_B}\mathrm{Mod}$, which we call \emph{Takeuchi's equivalence}, is given by the functors $\Phi$ and $\Psi$, the unit natural isomorphism
$
\unit: \F \to \Psi \circ \Phi(\F)$, defined by $\unit(f) = f_{(-1)} \otimes [f_{(0)}],
$
and the counit natural isomorphism 
$
\counit := \e \otimes \id: \Phi \circ \Psi(V) \to V.
$
The \emph{dimension} $\mathrm{dim}(\F)$ of an object $\F \in {}^A_B\mathrm{Mod}$ is the vector space dimension of $\Phi(\F)$.

Consider now the category $^A_B\textrm{Mod}_0$ whose objects are objects $\F$ in ${}^A_B\textrm{Mod}$ endowed with the right $B$-module structure 
\begin{align*}
fb := f_{(-2)}bS(f_{(-1)})f_{(0)}, & & \textrm{ for } f \in \F, \, b \in B. 
\end{align*}
It is instructive to note that any $b \in B$ acts on $A \square_{\pi_B} \Phi(\F)$ as left multiplication on the first tensor factor. Clearly, $^A_B\textrm{Mod}_0$ is equivalent to $^A_B\textrm{Mod}$, however $^A_B\textrm{Mod}_0$ comes equipped with a monoidal structure given by the tensor product $\otimes_B$. Moreover, with respect to the obvious monoidal structure on $^H\mathrm{Mod}$, Takeuchi's equivalence is easily endowed with the structure of a monoidal equivalence (see \cite[\textsection 4]{MMF2}). An immediate implication is that an object $\EE$ is invertible (that is, it is a relative line module) if and only if $\dim(\EE) = 1$.

\section{Principal Pairs of Quantum Homogeneous Spaces} \label{section:PPofQHS}

In this section we consider nested pairs of quantum homogeneous spaces $B \subseteq P$, which we consider as noncommutative generalisations of homogeneous fibrations (as discussed in the introduction). We show that $P$ can be presented as a relative Hopf module associated to a homogeneous fiber, and that $P$ is a type of generalised Hopf--Galois extension of $B$. Thus these pairs satisfy some of the key properties Brzezi\'{n}ski and Szyma\'{n}ski identified for \emph{noncommutative fibrations} \cite{TBWS}. We then consider the case where the fiber is a Hopf algebra, generalising homogeneous principle bundles as considered in classical geometry \cite[\textsection 2]{OVLie}. In this case $P$ is shown to be a principal comodule algebra with base $B$, and hence we recover the definition of a principal pair introduced by Mrozinski and the authors in \cite{BWGrass}. The new formulation offers a significant simplification of the original definition, and sets it in the context of classical homogeneous fibrations. We also show how Schneider's equivalence for principal comodule algebras interacts with the generalisation of Takeuchi's equivalence, and give necessary and sufficient criteria for our Hopf--Galois extensions to be non-cleft.

\subsection{Nested Pairs of Quantum Homogeneous Spaces}

Let $A$ be a Hopf algebra, and $B \subseteq P \subseteq A$ a nested pair of quantum homogeneous $A$-spaces. Note that $P$ is naturally an object in ${}^A_B\mathrm{Mod}$ so that the unit of Takeuchi's equivalence gives us the isomorphism 
$$
P \simeq A \, \square_{\pi_B} \pi_B(P)
$$
as objects in the category ${}^A_B\mathrm{Mod}$. Another easy observation is given by the following lemma, which will be used in the proof of Proposition \ref{prop:PPProp} for the special case of principal pairs. 

\begin{lem} \label{lem:FF}
For $A$ a Hopf algebra and $(B,P)$ a nested pair of quantum homogeneous $A$-spaces, $P$ is faithfully flat as a right $B$-module.
\end{lem}
\begin{proof}
Since every object of ${}^A_B\mathrm{Mod}$ is projective as a right $B$-module, $P$ is certainly flat as a right $B$-module. Now $P$ is \emph{faithfully}  flat as a right $B$-module if and only if $P \otimes_B \mathcal{M} \neq 0$, for all left $B$-modules $\mathcal{M}$. If $P \otimes_B \mathcal{M} = 0$, then we must have that $1_A \otimes_B \mathcal{M} = 0$, implying that $A \otimes_B \mathcal{M} = 0$, contradicting faithful flatness of $A$ as a right $B$-module. Thus we conclude that $P$ is faithfully flat as a right $B$-module. 
\end{proof}

For the special case of $A = P$, a nested pair obviously reduces to a quantum homogeneous space. The fact that any quantum homogeneous space is a Hopf--Galois extension of its space of coinvariants is generalised by the following result. 

\begin{prop} \label{prop:generalisedHG}
The map $\mathrm{can}:A \otimes_B A \to A \otimes \pi_B(A)$ restricts to an isomorphism
\begin{align*}
P \otimes_B P \to A \, \square_{\pi_P} \pi_B(P), 
\end{align*}
in ${}^A_P\mathrm{Mod}$, the category of $(A,P)$-relative Hopf modules. 
\end{prop}
\begin{proof}
Consider the vector space isomorphism:
\begin{align*}
\Phi_P(P \otimes_B P) \simeq \pi_B(P), & & [p' \otimes p] \mapsto \e(p')\pi_B(p).
\end{align*}
This gives us the commutative diagram
\begin{align*}
\xymatrix{ 
\Phi(P \otimes_B P) \ar[rrrr]^{\Phi_P(\mathrm{can}|_{P \otimes_B P})~~} \ar[rrd]_{\simeq} & & & & \Phi_P(A \, \square_{\pi_P} \pi_B(P)) \ar[lld]^{\counit} \\
& & \pi_B(P). & & 
}
\end{align*}
Thus the given map is an isomorphism in ${}^A_B\mathrm{Mod}$ as claimed.
\end{proof}

\subsection{A Generalisation of Takeuchi's Equivalence for Nested Pairs}

In this subsection we introduce a generalisation of Takeuchi's equivalence to the setting of nested pairs. We begin by introducing a generalisation of the category of relative Hopf modules. 

\begin{defn} For $A$ a Hopf algebra, and $(B,P)$ a nested pair of quantum homogeneous $A$-spaces, the objects of the category ${}^A_B\textrm{Mod}^{\pi_P}$ are those objects $\F \in {}^A_B\mathrm{Mod}$ endowed with a right $\pi_P(A)$-comodule structure $\Delta_{R,\pi_P}:\F \to \F \otimes \pi_P(A)$ giving $\F$ the structure of an $(A, \pi_P(A))$-bicomodule and such that $\Delta_{R,\pi_P}$ is a left $B$-module map. The morphisms are $(A, \pi_P(A))$-bicomodule and left $B$-module maps.
\end{defn}

For any $\F \in {}^A_B\mathrm{Mod}^{\pi_P}$, regarding $\F$ as an object in ${}^A_B\mathrm{Mod}$ gives the left $\pi_P(A)$-comodule $\Phi(\F)$. We introduce a right $\pi_P(A)$-coaction 
\begin{align*}
\Delta_{R,\pi_P}:\Phi(\F) \to \Phi(\F) \otimes \pi_P(A), & & [f] \mapsto [f_{(0)}] \otimes f_{(1)},
\end{align*}
which is well defined since $\Delta_{R,\pi_P}$ is by assumption a left $B$-module map. Operating on morphisms just as for Takeuchi's equivalence we get $(\pi_B(A),\pi_P(A))$-bicomodule maps, and so, we get a functor from ${}^A_B\mathrm{Mod}^{\pi_P}$ to ${}^{\pi_B}\mathrm{Mod}^{\pi_P}$. By abuse of notation we again denote this functor by $\Phi$.

In the opposite direction, for any $V \in {}^{\pi_B}\mathrm{Mod}^{\pi_P}$, regarding $V$ as an object in ${}^{\pi_B}\mathrm{Mod}$ we have the relative Hopf module $\Psi(V) \in {}^A_B\mathrm{Mod}$. We can endow $\Psi(V)$ with the right $\pi_P(A)$-coaction
\begin{align}
(\id \otimes \Delta_{R,\pi_P}): \Psi(V) = A \square_{\,\pi_B} V \to A \square_{\pi_B} V \otimes \pi_P(A), 
\end{align}
which is well defined since $V$ is a $(\pi_B(A),\pi_P(A))$-bicomdoule. Operating on morphisms just as for Takeuchi's equivalence, we get morphisms in ${}^A_B\mathrm{Mod}$, and so, we have defined a functor from ${}^{\pi_B}\mathrm{Mod}^{\pi_P}$ to ${}^A_B\mathrm{Mod}^{\pi_P}$. By abuse of notation we again denote this functor by $\Psi$.

\begin{prop}\label{prop:PPTak}
The components of Takeuchi's natural transformations $\unit$ and $\counit$ are right $\pi_P(A)$-comodule maps. Hence taken together with the functors $\Phi$ and $\Psi$, we have an adjoint equivalence between the categories ${}^A_B\mathrm{Mod}^{\pi_P}$ and ${}^{\pi_B}\mathrm{Mod}^{\pi_P}$. 
\end{prop} 
\begin{proof}
That each component of $\unit$ is a $\pi_P(A)$-comodule map follows from the fact that, for any $\F \in {}^A_B\mathrm{Mod}^{\pi_P}$, and any $f \in \F$, we have 
\begin{align*}
\xymatrix{ 
f \ar@{|->}[d]_{\Delta_{R,\pi_P}} \ar@{|->}[rrr]^{\unit} & & & f_{(-1)}\otimes [f_{(0)}] \ar@{|->}[d]^{\id_A\otimes\Delta_{R,\pi_P}}\\
f_{(0)}\otimes f_{(1)}\ar@{|->}[rrr]_{\unit \otimes \id_H} & & & f_{(-1)}\otimes [f_{(0)}] \otimes f_{(1)}.
}
\end{align*}
That the same holds true for $\counit$ follows from the fact that, for any $V \in {}^{\pi_B}\mathrm{Mod}^{\pi_P}$, and any $\sum_i a_i \otimes v_i \in A \square_{\,\pi_B} V$, we have
\begin{align*}
\xymatrix{ 
\left[\sum_i a_i\otimes v_i\right] \ar@{|->}[d]_{\Delta_{R,\pi_P}} \ar@{|->}[rrr]^{\counit} & & & \sum_i \e (a_i) v_i \ar@{|->}[d]^{\Delta_{R,\pi_P}}\\
\left[\sum_i a_i\otimes (v_{i})_{(0)}\right]\otimes (v_{i})_{(1)} \ar@{|->}[rrr]_{\counit \otimes \id_H} & & & \sum_i \e (a_i) (v_{i})_{(0)}\otimes (v_{i})_{(1)}.
}
\end{align*}
Thus $\Phi$, $\Psi$, $\unit$, and $\counit$ give an adjoint equivalence as claimed.
\end{proof}

\begin{remark}
In \cite[Proposition 3.5]{BWGrass} an alternative generalisation of Takeuchi's equivalence was considered for principal pairs (see Remark \ref{rem:PPOldandNew} below). As a careful reading of the proof will confirm, the argument can be directly extended to the setting of nested pairs of quantum homogeneous spaces.
\end{remark}

\subsection{Some Consequences of the Equivalence}

For any nested pair of quantum homogeneous subspaces $(B,P)$, we have the surjective Hopf algebra map
\begin{align*}
\pi_{B,P}: \pi_B(A) \to \pi_P(A), & & \pi_B(a) \mapsto \pi_P(a).
\end{align*}
Thus $\pi_P(A)$ is canonically a  quantum subgroup of $\pi_B(A)$. This gives us a right $\pi_P(A)$-coaction on $\pi_P(A)$, with respect to which $\pi_B(A)$ has the structure of an object in ${}^{\pi_B(A)}\mathrm{Mod}^{\pi_B(A)}$. 
The following proposition gives an easy but important consequence of Takeuchi's equivalence relating $P$, $B$, and the space of coinvariants of this right coaction. 

\begin{prop} \label{prop:NestedPairs}
For $(B,P)$ a nested pair of quantum homogeneous $A$-spaces:
\begin{enumerate}
\item an isomorphism in the category ${}^A_B\mathrm{Mod}$ is given by
\begin{align*}
P \to A \, \square_{\,\pi_B}\big(\pi_B(A)^{\co(\pi_P(A))}\big), & & p \mapsto p_{(1)} \otimes \pi_B(p_{(2)}),
\end{align*}

\item we have the identity 
$$
\pi_B(P) = \pi_B(A)^{\co(\pi_P(A))}.
$$

\end{enumerate}
\end{prop}
\begin{proof}
~\\
1. ~ The unit of Takeuchi's equivalence gives a right $\pi_P(A)$-comodule isomorphism $\unit:A \to A \, \square_{\,\pi_B} \pi_B(A)$, thus we see that 
$$
P = A^{\co(\pi_P(A))} \simeq \left(A \, \square_{\,\pi_B} \pi_B(A)\right)^{\co(\pi_P(A))} = A \, \square_{\,\pi_B} \left( \pi_B(A)\right)^{\co(\pi_P(A))}.
$$
Thus the claimed isomorphism is given by the unit of the equivalence.

2. ~ Operating by the functor $\Phi$ on the isomorphism given in 1 and then composing $\Phi(\unit)$ with the counit $\counit$ of the equivalence, we get the isomorphism
\begin{align*}
\counit \circ \Phi(\unit): \Phi_B(P) = \pi_B(P) \to \pi_B(A)^{\co(\pi_P(A))}, & & \pi_B(p) \mapsto \pi_B(p),
\end{align*}
and hence the claimed identity. 
\end{proof}

\begin{remark} \label{remark:noncommutativefibration}
Motivated by Lemma \ref{lem:FF}, Proposition \ref{prop:generalisedHG}, and Proposition \ref{prop:NestedPairs}, we think of the triple of algebras
$$
B \hookrightarrow P \twoheadrightarrow \pi_B(A)^{\co(\pi_P(A))}
$$
as a \emph{noncommutative homogeneous fibration}, and refer to it as such, at least informally. Moreover we refer to $\pi_B(A)^{\co(\pi_P(A))}$ as the \emph{fiber} of the fibration.
\end{remark}

\subsection{Principal Pairs}

We now arrive at the main goal of this section: a new and simplified version of the definition of a principal pair, as first introduced in \cite[Definition 3.2]{BWGrass}. The precise relationship between the two definitions is discussed in the remark below. The definition can be viewed as a noncommutative generalisation of a fibration of homogeneous space over a homogeneous space
\begin{align*}
 G/H \to G/K, 
\end{align*}
where $G$ is a group, and $H \subseteq K \subseteq G$ are subgroups such that $H$ is normal in $K$. (See \cite[\textsection 2]{OVLie} for a more detailed discussion.)

\begin{defn}
For a Hopf algebra $A$, a \emph{principal $A$-pair} is a pair $(B,P)$ of nested quantum homogeneous $A$-spaces such that 
$
\pi_B(P) = \pi_B(A)^{\co(\pi_P(A))}
$
is a Hopf subalgebra of $\pi_B(A)$.
\end{defn}

We will show how to construct a principal comodule algebra from any principal pair, beginning with a lemma introducing a right $\pi_B(P)$-coaction on $P$.

\begin{lem}
For a principal $A$-pair $(B,P)$, it holds that 
$$
\Delta_{R,\pi_B}(P) \sseq P \otimes \pi_B(P).
$$
Hence $\Delta_{R,\pi_B}$ restricts to a coaction 
$
\Delta_{R,\pi_B}: P \to P \otimes \pi_B(P)
$
and $B = P^{\co(\pi_B(P))}$.
\end{lem}
\begin{proof}
Let us first show that $P$ is a right $\pi_B(A)$-subcomodule of $A$. It follows from the isomorphism given in Proposition \ref{prop:NestedPairs} that $P$ is a right $\pi_B(A)$-subcomodule of $A$ if and only if $A \square_{\,\pi_B} \pi_B(P)$ is a right subcomodule of $A \square_{\,\pi_B} \pi_B(A)$. However, this is a direct consequence of our assumption that $\pi_B(P)$ is a Hopf subalgebra of $\pi_B(A)$. The fact that $B = P^{\co(\pi_B(P))}$ is evident.
\end{proof}

\begin{prop} \label{prop:PPProp}
The pair $(P,\Delta_{R,\pi_B(P)})$ is a principal comodule algebra.
\end{prop} 
\begin{proof}
It follows from Proposition \ref{prop:generalisedHG} that the canonical map induces an isomorphism 
$$
P \otimes_B P \to A \square_{\,\pi_P} \pi_B(P).
$$
However, $\pi_B(P)$ is a Hopf subalgebra of $\pi_B(A)$ by assumption, and so, it is trivial as a left $\pi_P(A)$-comodule, meaning that 
$$
A \square_{\,\pi_P} \pi_B(P) = A \otimes \pi_B(P).
$$
Thus $\mathrm{can}$ is an isomorphism, which is to say, $P$ is a Hopf--Galois extension of $B$. 

Faithful flatness of $P$ as a right $B$-module follows from Proposition \ref{lem:FF}. Recalling that $P$ is a Hopf--Galois extension of $B$, it now follows from \cite[Theorem I]{Schneider90} that $P$ is also faithfully flat as a left $B$-module. Thus the pair $(P,\Delta_{R,\pi_B(P)})$ is a principal comodule algebra as claimed.
\end{proof}

\begin{remark}\label{rem:PPOldandNew}
From Proposition \ref{prop:PPProp} it is clear that any principal $A$-pair $(B,P)$ the pair
$$
(\pi_P:A \to \pi_P(A), \pi_B:A \to \pi_B(A))
$$
is a principal pair in the original sense of \cite[Definition 3.2]{BWGrass}. Conversely, we see that for principal pair in the original sense of \cite{BWGrass} will give a principal pair in the sense of this paper. A subtle point is that the definition of \cite{BWGrass} assumes explicit choices of quantum subgroups $\pi_H:A \to H$ and $\pi_K:A \to K$, where no such choice is assumed in this paper. 
\end{remark}


\subsection{Schneider's Equivalence for Principal Pairs} \label{subsection:SchneidersEquiv}

In this subsection we look at Schneider's equivalence for the special case of principal pairs. We begin by introducing a new category. 

\begin{defn}
For $(B,P)$ a principal pair, the objects of the category ${}^A_P\mathrm{Mod}^{\pi_{B}(P)}$ are objects $\F$ in ${}_P\mathrm{Mod}^{\pi_B(P)}$ endowed with a left $A$-comodule structure $\Delta_{L,A}:\F \to A \otimes \F$ such that $\F$ is an $(A,P)$-relative Hopf module, and an $(A,\pi_B(P))$-bicomodule. The morphisms are left $A$-comodule, right $\pi_B(P)$-comodule, left $P$-module maps. 
\end{defn}

For any $\F \in {}^A_P\mathrm{Mod}^{\pi_B(P)}$, regarding $\F$ as an object in ${}_P\mathrm{Mod}^{\pi_B(P)}$, we consider the object $\F^{\co(\pi_B(P))} \in {}_B\mathrm{Mod}$. We introduce a left $A$-coaction 
\begin{align*}
\Delta_{L,A}: \F^{\co(\pi_B(P))} \to \F^{\co(\pi_B(P))}, & & f \mapsto f_{(-1)} \otimes f_{(0)},
\end{align*}
which is well defined because $\F$ is an $(A,\pi_B(P))$-bicomodule. This gives a functor from ${}^A_P\mathrm{Mod}^{\pi_B(P)}$ to ${}^A_B\mathrm{Mod}$, operating on morphisms just as for Schneider's equivalence, and which by abuse of notation we again denote by $(-)^{\co(\pi_B(P))}$. 

In the opposite direction, for any $\N \in {}^A_B\mathrm{Mod}$, we can endow $P \otimes_B \N$ with the tensor product left $A$-coaction, and the right $T$-coaction $\Delta_R \otimes \id$. This gives us a functor 
$$
P \otimes_B - : {}^A_B\mathrm{Mod} \to {}^A_P\mathrm{Mod}^{\pi_B(P)},
$$
which acts on morphisms just as for Schneider's equivalence.

\begin{prop}
The functors $(-)^{\co(\pi_B(P))}$ and $A \otimes_B -$ give an adjoint equivalence between ${}^A_P\mathrm{Mod}^{\pi_B(P)}$ and ${}^A_B\mathrm{Mod}$.
\end{prop}
\begin{proof}
To establish the proposition, we need only show that the unit and counit morphisms of Schneider's equivalence are left $A$-comodule maps. For the unit $\unit$ this follows from the commutativity of the following diagram
\begin{align*}
\xymatrix{ 
m \ar@{|->}[d]_{\Delta_{L,A}} \ar@{|->}[rr]^{\unit} & & 1\otimes m \ar@{|->}[d]^{\Delta_{L,A}\otimes\Delta_{L,A}}\\
m_{(0)}\otimes m_{(1)}\ar@{|->}[rr]_{ \id_A \otimes \unit} & & m_{(0)}\otimes 1\otimes m_{(1)}.
}
\end{align*}
while for the counit $\counit$ we have
\begin{align*}
\xymatrix{ 
p \otimes n \ar@{|->}[d]_{\Delta_{L,A}} \ar@{|->}[rr]^{\counit} & & pn \ar@{|->}[d]^{\Delta_{L,A}}\\
p_{(0)}n_{(0)}\otimes p_{(1)} \otimes n_{(1)} \ar@{|->}[rr]_{\id_A\otimes \counit} & & p_{(0)}n_{(0)}\otimes p_{(1)} n_{(1)}.
}
\end{align*}
Thus the functors $(-)^{\co(\pi_B(P))}$ and $A \otimes_B -$ give an adjoint equivalence as claimed.
\end{proof}

Combining this with Takeuchi's equivalence we immediately get the following corollary, showing that ${}^A_P\mathrm{Mod}^{\pi_B(P)}$ is in fact equivalent to the category of $\pi_B(P)$-comodules.

\begin{cor}\label{Cor:ScnEq}
An equivalence between the categories ${}^A_P\mathrm{Mod}^{\pi_B(P)}$ and ${}^{\pi_B}\mathrm{Mod}$ is given by the functors $\Phi \circ (-)^{\co(\pi_B(P))}$ and $( A \otimes_B -) \circ \Psi$.
\end{cor}

\section{Group Algebra Principal Pairs and Relative Line Modules}\label{section:GAPP}

In this section we introduce a special type of principal pair called a \emph{group algebra principal pair}. This is motivated by the classical special case of a homogeneous principal fibration of the form $G/L^{\mathrm{s}} \to G/L$, where $L$ is a reductive subgroup, and $L^{\mathrm{s}}$ is the semisimple part of $L$. In this special case, $P$ decomposes into a direct sum of relative line modules, and moreover, every relative line module over the base $B$ is of this form.

\subsection{Group Algebra Principal Pairs}

For any Hopf algebra $H$ we denote by $\Gamma_H$ the group of grouplike elements of $H$. Moreover, we denote by $\mathbb{C}\Gamma_H$ the group algebra of $\Gamma_H$. Note that $\mathbb{C}\Gamma^+_H$ is a subalgebra of $\mathbb{C}\Gamma_H$, closed under the antipode, and indeed a two-sided coideal. Thus the quotient space
$$
K_H := H/\langle \mathbb{C}\Gamma^+_H \rangle = H/\langle g - 1 \,|\, g \in \Gamma_H \rangle
$$
inherits from $H$ the structure of a Hopf algebra. We denote by $\proj_{K}:H \to K_H$ the canonical Hopf algebra projection.

Let us now specialise to the case where the Hopf algebra $H$ is cosemisimple with Peter--Weyl decomposition $H \simeq \bigoplus_{\alpha \in \widehat{H}} H_{\alpha}$. Consider the following equivalence relation on $\widehat{H}$. For $\alpha, \, \alpha' \in \widehat{H}$, we say that $\alpha \sim \alpha'$ if 
$$
\mathbb{C}\Gamma_H H_{\alpha} = \mathbb{C} \Gamma_H H_{\alpha'}. 
$$
We denote by $\widehat{H}_{\Gamma}$ the associated set of equivalence classes of elements of $\widehat{H}$, and write
$$
H_{[\alpha]} := \mathbb{C} \Gamma_H H_{\alpha}.
$$
We call the coarser decomposition 
$$
H \simeq \bigoplus_{[\alpha] \in \widehat{H}_{\Gamma}} H_{[\alpha]}
$$
the \emph{$\Gamma$-decomposition} of $H$. Note that this is a decomposition of $H$ as left $\mathbb{C}\Gamma^+_H$ module.

\begin{lem} \label{lem:GPW}
Let $H = \bigoplus_{\alpha \in \widehat{H}} H_{\alpha}$ be a cosemisimple Hopf algebra, and assume that $\mathbb{C}\Gamma_H \subseteq Z_H$, where $Z_H$ denotes the center of $H$. A direct sum decomposition of $K_H$ into subcoalgebras is given by 
\begin{align} \label{eqn:GPW}
K_H \simeq \bigoplus_{[\alpha] \in \widehat{H}_{\Gamma}}\proj_K(H_{[\alpha]})= \bigoplus_{[\alpha] \in \widehat{H}_{\Gamma}} K_{[\alpha]},
\end{align}
where for each $[\alpha] \in \widehat{H}_{\Gamma},$ we denote $K_{H_{[\alpha]}}:=\proj_K(H_{[\alpha]}).$ 
\end{lem}
\begin{proof}
Note first that since $\Gamma_H$ is central in $H$, the ideal generated by $\Gamma_H^+$ is homogeneous with respect to the $\Gamma$-Peter--Weyl decomposition, and hence the $\Gamma$-Peter--Weyl decomposition descends to the quotient. Since, for any $\alpha \in \widehat{H}$, we have 
$$
\proj_K(\mathbb{C}\Gamma_H H_{\alpha}) = \proj_K(H_{\alpha}),
$$
we see that the decomposition is given explicitly by \eqref{eqn:GPW}. Finally, note that projection onto the identity component defines a Haar functional proving that $K_H$ is cosemisimple. 
\end{proof}

\begin{thm} \label{thm:GAPP}
Let $A$ and $H$ be Hopf algebras, $\pi_H:A \to H$ a surjective Hopf algebra map, and $\proj_{K}:H \to K_H$ the canonical projection. If $H$ is cosemisimple and $\Gamma_H \subseteq Z_H$ then
\begin{align} \label{eqn:GAPPIdentity}
\pi_H(P) = \pi_H(A)^{\co(K_H)} = \mathbb{C}\Gamma_H.
\end{align}
Moreover, the pair 
$$
\big(B := A^{\co(H)}, \, P := A^{\co(K_H)}\big)
$$
is a principal $A$-pair. 
\end{thm}
\begin{proof}
Since $H$ is cosemisimple, the space of coinvariants associated to $\pi_H$ is a quantum homogeneous space. Moreover, since Lemma \ref{lem:GPW} tells us that $K_H$ is also cosemisimple, the space of coinvariants associated to $\pi_{K_H}$ is also a quantum homogeneous space. Thus the pair $(B,P)$ is a nested pair of quantum homogeneous $A$-spaces.

Let us now establish \eqref{eqn:GAPPIdentity}. The first equality follows directly from Lemma \ref{prop:NestedPairs}. For the second equality, note first that the construction of $K_H$ implies the inclusion
\begin{align*}
\Delta_{R,K_H}(H_{\alpha}) \subseteq H_{\alpha} \otimes K_{H,\alpha}. 
\end{align*}
Thus we see that a non-zero element of $H_{\alpha}$ is right $K_H$-coinvariant if and only if $\alpha$ is equivalent to the trivial representation. Now since $H^{\co(K_H)}$ is homogeneous with respect to the Peter--Weyl decomposition of $H$, we have that $H^{\co(K_H)} = \mathbb{C}\Gamma_H$. So in particular, $H^{\co(K_H)}$ is a Hopf subalgebra of $H$, which is to say, the nested pair is a principal pair.
\end{proof}

We now give a name to the principal pairs considered in the theorem, which we find convenient to present as a separate definition.

\begin{defn}
Let $A$ be a Hopf algebra, $H$ a cosemisimple Hopf algebra such that $\Gamma_H \subseteq Z_H$, and $\pi:A \to H$ a surjective Hopf algebra map. We call the associated principal pair $(B = A^{\co(H)},\, P = A^{\co(K_H)})$ a \emph{group algebra principal pair}.
\end{defn}

\subsection{Relative Line Modules and Group Algebra Principal Pairs}

A $\mathbb{C}\Gamma_H$-comodule algebra structure on $P$ is equivalent to a $\Gamma_H$-algebra grading of $P$, which we denote by 
\begin{align} \label{eqn:PDecomp}
P \simeq \bigoplus_{\gamma \in \Gamma_H} \EE_{\gamma}.
\end{align}

\begin{prop} \label{lem:GAPP}
 For a group algebra principal pair
$$
\left(B = A^{\co(H)}, P = A^{\co(K_H)}\right)
$$
we have the following results:
\begin{enumerate}
 \item The $\Gamma_H$-grading of $P$ is \emph{strong}, which is to say, $\EE_{\gamma}\EE_{\gamma'} = \EE_{\gamma\gamma'}$, for all $\gamma,\gamma' \in \Gamma_H$.
 \item Each $\EE_{\gamma}$ is a relative line module over $B = \EE_0$, in particular  $\Phi(\EE_{\gamma}) = \pi_B(\EE_{\gamma}) = \mathbb{C}\gamma$.
 \item Every relative line module $\EE$ over $B$ is isomorphic to $\EE_{\gamma}$, for some $\gamma \in \Gamma_H$.
 \item The decomposition in \eqref{eqn:PDecomp} can equivalently be  described as the unique decomposition of $P$ into its simple sub-objects.
\end{enumerate} 
\end{prop}
\begin{proof}
~~~\\
1. Since $P$ is a Hopf--Galois extension of $B$, it necessarily follows that $P$ is strongly graded (see \cite[Theorem 8.1.7]{Monty}). 

2. It is clear from the construction of $\EE_{\gamma}$ that $\Phi(\EE_{\gamma}) = \pi_B(\EE_{\gamma}) = \mathbb{C}\gamma$, and so, $\EE_{\gamma}$ is a relative line module.

3. Every simple $H$-comodule is embeddable into $H$. Since $\Phi_B(\EE)$ is a one-dimensional $H$-comodule by assumption,  it is embeddable into $\mathbb{C}\Gamma_H$. Thus we must have that an isomorphism $\Phi(\EE) \simeq \Phi(\EE_\gamma)$, for some $\gamma \in \Gamma_H$, or equivalently it must hold that $\EE \simeq \EE_\gamma$. 

4. Since each $\EE_{\gamma}$ is simple, the decomposition is indeed a decomposition of $P$ into simple subobjects. The relative line modules are all distinct since $\Phi(\EE_{\gamma})$ is not isomorphic to $\Phi(\EE_{\gamma'})$ when $\gamma \neq \gamma'$. Thus we see the decomposition of $P$ into relative line modules is the unique decomposition into simple sub-objects.
\end{proof}

\begin{cor}\label{cor:cosetrep}
It holds that $\Phi(\EE_{\alpha}) = \mathbb{C}[e]$, if and only if $e \not \in \EE_k^+ = \EE_k \cap A^+$\!.
\end{cor}
\begin{proof}
Since the $\Gamma_H$-grading is strong, there exists a sum $\sum_{i} e_i \otimes e'_i$ in $\EE_{-\gamma} \otimes_B \EE_{\gamma}$ such that 
$
\sum_i e_ie_i' = 1,
$
for any $\gamma \in \Gamma_H$. Thus, for any $e \in \EE^+_\gamma$, we have 
$$
[e] = \sum_i [ee_ie_i'] = 0,
$$
where the second equality follows from the fact that $ee_i \in B^+$.

Consider next an element $e \notin \EE_k^+$. The coset $[e] = 0$ if and only only if $e = \sum_i b_ie_i$, for some $b_i \in B^+, \, e_i \in \EE_{\gamma}$. But in this case 
$$
\e(e) = \sum_i \e(b_i)\e(e_i) = 0,
$$
contradicting our assumption that $e \not \in \EE^+_{\gamma}$. Thus $\Phi(\EE_{\gamma})$ is spanned by $[e]$.
\end{proof}


\begin{eg} \label{eg:degGAPP}
Let $A$ and $H$ be Hopf algebras, and $\pi:A \to H$ a surjective Hopf algebra map. If $H$ is abelian and satisfies $H = \mathbb{C}\Gamma_H$, then $\pi:A \to H$ clearly satisfies the requirements of Theorem \ref{thm:GAPP}, and so, we have a group algebra principal pair. In this case $K_H \simeq \mathbb{C}1$, implying that $A = P^{\co(K_H)}$, and giving us the strong $\Gamma_H$-algebra grading 
$$
A \simeq \bigoplus_{\gamma \in \Gamma_H} \EE_{\gamma},
$$
in the category ${}^A_B\mathrm{Mod}_0$. Thus every object in ${}^A_B\mathrm{Mod}_0$ is a direct sum of relative line modules. This is the situation for the full quantum flag manifolds $\OO_q(G/\mathbb{T}^{r})$ presented in \textsection \ref{section:QFMs} below.
\end{eg}

\begin{eg}
The case where where $\Gamma_H = 1$ again gives a group algebra principal pair, but in this case $K_H$ is clearly isomorphic to $H$, meaning that 
$$
P = A^{\co(K_H)} = A^{\co(H)} = B,
$$
which is to say, the Hopf--Galois extension is trivial.
\end{eg}

\subsection{Schneider's Equivalence for Group Algebra Principal Pairs}\label{subsection:SchneideGAPP}

We finish this section with some observations about Schneider's equivalence in the special setting of group algebra principal pairs. Recalling that comodules over $\mathbb{C}\Gamma_H$ are the same as graded modules over the group $\Gamma_H$, we see that Schneider's equivalence gives us
$$
{}_P \mathrm{Mod}^{\mathrm{gr}(\Gamma_H)} \simeq {}_B\mathrm{Mod}.
$$
Moreover, the left $A$-covariant version of Schneider's equivalence gives us 
\begin{align} \label{eqn:ASchneiderGAPP}
{}^A_P \mathrm{Mod}^{\mathrm{gr}(\Gamma_H)} \simeq {}^A_B\mathrm{Mod} \simeq {}^{K_H}\mathrm{Mod},
\end{align}
which is to say, $\Gamma_H$-graded $(A,P)$-relative Hopf modules are equivalent to $K_H$-comodules.

\section{Quantum Flag Manifolds} \label{section:QFMs}

In this section we give a principal pair description of the quantum flag manifolds. To set notation, we begin by recalling the necessary definitions and results of the Drinfeld--Jimbo quantum enveloping algebras and quantum coordinate algebras, as well as the quantum flag manifolds $\OO_q(G/L_S)$ and their associated quantum Poisson homogeneous spaces $\OO_q(G/L^{\mathrm{s}}_S)$ introduced by Dijkhuizen and Stokman \cite{DijkStok}. We then present the pair $\OO_q(G/L_S)$ and $\OO_q(G/L^{\mathrm{s}}_S)$ as a group algebra principal pair. This result is then in turn used to establish a number of categorical equivalences, to give a complete description of the relative line modules over $\OO_q(G/L_S)$, and to construct an explicit principal $\ell$-map.

\subsection{Drinfeld--Jimbo Quantum Groups} \label{subsection:DJQG}

Let $\frak{g}$ be a finite-dimensional complex semisimple Lie algebra of rank $r$. We fix a Cartan subalgebra $\frak{h}$ and choose a set of simple roots $\Pi = \{\alpha_1, \dots, \alpha_r\}$ for the corresponding root system in $\frak{g}^*$, where $\frak{g}^*$ denotes the linear dual of $\frak{g}$. We denote by $(\cdot,\cdot)$ the symmetric bilinear form induced on $\frak{h}^*$ by the Killing form of $\frak{g}$, normalised so that any shortest simple root $\a_i$ satisfies $(\a_i,\a_i) = 2$. The Cartan matrix $(a_{ij})_{ij}$ of $\frak{g}$ is defined by 
$
a_{ij} := \big(\alpha_i^{\vee},\alpha_j\big),
$
where $\alpha_i^{\vee} := 2\a_i/(\a_i,\a_i)$.

Let $q \in \bR$ such that $q \neq -1,0,1$, and denote $q_i := q^{(\alpha_i,\alpha_i)/2}$. The {\em Drinfeld--Jimbo quantised enveloping \algn} $U_q(\frak{g})$ is the noncommutative associative algebra generated by the elements $E_i, F_i, K_i$, and $K^{-1}_i$, for $ i=1, \ldots, r$, subject to the relations 
\begin{align*}
 K_iE_j = q_i^{a_{ij}} E_j K_i, ~~~~ K_iF_j= q_i^{-a_{ij}} F_j K_i, ~~~~ K_i K_j = K_j K_i, ~~~~ K_iK_i^{-1} = K_i^{-1}K_i = 1,\\
 E_iF_j - F_jE_i = \d_{ij}\frac{K_i - K\inv_{i}}{q_i-q_i^{-1}}, ~~~~~~~~~~~~~~~~~~~~~~~~~~~~~~~~~~~~~~~~~
\end{align*}
along with the quantum Serre relations which we omit (see \cite[\textsection 6.1.2]{KSLeabh} for details).
The formulae 
\begin{align*}
\DEL(K_i) = K_i \oby K_i, & & \DEL(E_i) = E_i \oby K_i + 1 \oby E_i, & & \DEL(F_i) = F_i \oby 1 + K_i\inv \oby F_i 
\end{align*}
define a Hopf algebra structure on $U_q(\frak{g})$, satisfying $\e(E_i) = \e(F_i) = 0$, and $\e(K_i) = 1$.

\subsection{Type-$1$ Representations}

The set of {\em fundamental weights} $\{\varpi_1, \dots, \varpi_r\} \sseq \frak{h}^*$ of $\frak{g}$ are defined by $\big(\alpha_i^{\vee}, \varpi_j\big) = \delta_{ij}$, for all $i,j = 1, \dots, r$. We denote by ${\mathcal P}$ the {\em integral weight lattice} of $\frak{g}$, which is to say the $\mathbb{Z}$-span of the fundamental weights. Moreover, we denote by ${\mathcal P}^+$ the cone of {\em dominant integral weights}, which is to say the $\mathbb{Z}_{\geq 0}$-span of the fundamental weights. 


A vector $v\in V$, for any $U_q(\frak{g})$-module $V$, is called a \emph{weight vector} of \emph{weight}~$\mathrm{wt}(v) \in \mathcal{P}$ if
\begin{align}\label{eq:Kweight}
K_i \triangleright v = q^{(\alpha_i, \mathrm{wt}(v))} v, & & \textrm{ for all } i=1,\ldots,r.
\end{align}
For each $\lambda \in\mathcal{P}^+$ there exists an irreducible finite-dimensional $U_q(\frak{g})$-module $V_\lambda$, uniquely defined by the existence of a weight vector $v \in V_\lambda$, of weight $\lambda$ satisfying
$
 E_i \triangleright v =0, 
 \text{for all $i=1,\ldots,r$.}
$
We call such a vector $v$ a {\em highest weight vector}. The vector $v$ is uniquely determined up to scalar multiple. We call any finite direct sum of such $U_q(\frak{g})$-representations a {\em type-$1$ representation}, and we denote by ${}_{U_q(\frak{g})}\mathbf{type}_1$ the full subcategory of ${U_q(\frak{g})}$-modules whose objects are finite  sums of type-1 modules. Each type-$1$ module $V_{\lambda}$ decomposes into a direct sum of \emph{weight spaces}, which is to say, those subspaces of $V_{\lambda}$ spanned by weight vectors. We now choose, once and for all, a weight basis $\{v_i\}_{i=1}^{N_{\lambda}}$, for each irreducible representation $V_{\lambda}$, for $N_{\lambda} := \dim(V_{\lambda})$, which is to say a vector space basis of $V_{\lambda}$ whose elements are weight vectors. We label our basis so that $v_{N_{\lambda}}$ is the highest weight basis vector. 

Since $U_q(\frak{g})$ has an invertible antipode, we have an equivalence between the category of left $U_q(\frak{g})$-modules and the category of right $U_q(\frak{g})$-modules, induced by the antipode. For any finite-dimensional left $U_q(\frak{g})$-module $V$, we denote by $V^*$ the $\mathbb{C}$-linear dual of $V$, endowed with its right \mbox{$U_q(\frak{g})$-module} structure. With respect to the equivalence of left and right $U_q(\frak{g})$-modules, the left module corresponding to $V^*_{\mu}$ is isomorphic to $V_{-w_0(\mu_S)}$, where $w_0$ denotes the longest element in the Weyl group of $\frak{g}$. When discussing a specific irreducible representation $V_{\lambda}$, we usually find it convenient to denote by $\{f_i\}_{i=1}^{N_{\lambda}}$ the basis of $V_{-w_0(\lambda)}$ dual to $\{v_i\}_{i=1}^{N_{\lambda}}$.

\subsection{Quantum Levi Subalgebras}

For $S$ a proper subset of simple roots, consider the Hopf subalgebra of $U_q(\frak{g})$ given by 
\begin{align*}
U_q(\frak{l}_S) := \big< K_i, E_j, F_j \,|\, i = 1, \ldots, r; \, j \in S \big>.
\end{align*} 
We denote $S^c := \Pi \backslash S$, and represent $S^c$ graphically by coloured nodes in the Dynkin diagram of $\frak{g}$. The definition of a type-$1$ module for $U_q(\frak{g})$ has an evident analogue for the subalgebra $U_q(\frak{l}_S)$, giving us the category ${}_{\uslevi}\mathbf{type}_1$. 

Classically the algebra $\frak{l}_S$ is reductive, and hence decomposes into a direct sum of a semisimple part and a central part. This means that in the quantum setting we are motivated to consider the subalgebra 
\begin{align*}
\usslevi := \big< K_i, E_i, F_i \,|\, i \in S \big> \sseq U_q(\frak{l}_S).
\end{align*} 
Just as in the classical case, the sublattice
$$
\mathcal{P}^+_{S} + \mathcal{P}_{S^c} \subseteq \mathcal{P},  
$$
labels the one-dimensional type-$1$ representations of $\uslevi$, where we have denoted 
\begin{align*}
\mathcal{P}^+_{S} :=  & \, \Big\{ \lambda \in \mathcal{P}^+ \,|\, \lambda = \sum^{\,}_{s \in S}\, a_s \varpi_s, \textrm{ for some } a_s \in \mathbb{Z}_{\geq 0} \Big\},\\
\mathcal{P}_{S^{\mathrm{c}}} := & \, \Big\{ \lambda \in \mathcal{P} \,|\, \lambda = \sum^{\,}_{x \in S^{\mathrm{c}}}\, a_x \varpi_x, \textrm{ for some } a_x \in \mathbb{Z} \Big\}.
\end{align*}
Since $\frak{l}_S^{\,\mathrm{s}}$ is a semisimple Lie algebra, it admits no non-trivial one-dimensional representations, which means that the Hopf algebra $U_q(\frak{l}_S^{\,\mathrm{s}})$ admits no non-trivial $1$-dimensional representations. From this it is clear that the one-dimensional representations of $\uslevi$ are labelled by the sublattice $\mathcal{P}_{S^c}$.

\subsection{Quantum Coordinate Algebras and the Quantum Flag Manifolds}

Let $V$ be a finite-dimensional $U_q(\frak{g})$-module, $v \in V$, and $f \in V^*$. Consider the function $c^{\textrm{\tiny $V$}}_{f,v}:U_q(\frak{g}) \to \bC$ defined by $c^{\textrm{\tiny $V$}}_{f,v}(X) := f\big(Xv\big)$.
%
The {\em coordinate ring} of $V$ is the subspace
\begin{align*}
C(V) := \text{span}_{\mathbb{C}}\!\left\{ c^{\textrm{\tiny $V$}}_{f,v} \,| \, v \in V, \, f \in V^*\right\} \sseq U_q(\frak{g})^{\circ}.
\end{align*}
A Hopf subalgebra of $U_q(\frak{g})^\circ$ is given by 
\begin{align*}
\O_q(G) := \bigoplus_{\mu \in \mathcal{P}^+} C(V_{\mu}).
\end{align*}
We call $\OO_q(G)$ the {\em quantum coordinate algebra of $G$}, where $G$ is the compact simply-connected Lie group having $\frak{g}$ as its complexified Lie algebra. Note that the Hopf algebra $\O_q(G)$ is cosemisimple by construction. 

Consider now the coideal subalgebra of $\uslevi$-invariants
\begin{align*}
\O_q\big(G/L_S\big) := {}^{U_q(\frak{l}_S)}\O_q(G),
\end{align*} 
with respect to the natural left $U_q(\frak{g})$-module structure on $\OO_q(G)$. Just as for $U_q(\frak{g})$, we can talk about the \emph{type-1 dual} $\OO_q(L_S)$ of $\uslevi$.  Equivalently, $\OO_q(G/L_S)$ can be described as the space of coinvariants of the right $\OO_q(L_S)$-coaction associated to the Hopf algebra map 
$$
\rho_{S}:\OO_q(G)  \to  \OO_q(L_S),
$$ 
given by restriction of domains. Indeed, since any finite-dimensional $\uslevi$-module can be embedded into a finite-dimensional $U_q(\frak{g})$-module, this map is surjective, and 
$$
\pi_{\OO_q(G/L_S)} \simeq \OO_q(L_S). 
$$
Since $\OO_q(L_S)$ is cosemisimple by construction, $\OO_q(G/L_S)$ is a quantum homogeneous space. We call it the {\em quantum flag manifold associated to} $S$. 

Since any finite-dimensional $\uslevi$-module can be embedded into a finite-dimensional $U_q(\frak{g})$-module,


\subsection{The Associated Quantum Poisson Homogeneous Spaces} \label{Subsection:QPHS}

Consider now the coideal subalgebra of $\usslevi$-invariants
\begin{align*}
\O_q\big(G/L^{\,\mathrm{s}}_S\big) := {}^{\usslevi}\O_q(G),
\end{align*} 
which we call the \emph{quantum homogeneous Poisson space associated} to $S$. Just as for the quantum flag manifolds, we have an isomorphism 
$$
\pi_{\OO_q(G/L^{\mathrm{\,s}}_S)}(\OO_q(G)) \simeq \OO_q(L^{\mathrm{\,s}}_S),
$$
where $\OO_q(L^{\mathrm{\,s}}_S)$ is the type-$1$ dual of $\usslevi$. Since $\sslevi$ is cosemisimple by construction, $\OO_q(G/L_S)$ is a quantum homogeneous space.

As shown in \cite[\textsection Theorem 4.1]{Stok}, a set of generators for $\OO_q(G/L_S^{\,\mathrm{s}})$ is given by 
\begin{align} \label{eqn:QPHSGens}
z^{\varpi_x}_{i} := c^{V_{\varpi_x}}_{f_i,v_{N_x}}, & & \overline{z}_i^{\,\varpi_x} := c^{V_{-w_0(\varpi_x)}}_{v_i,f_{N_x}} & & \text{ for }  i = 1, \dots, N_x, \textrm{ and } x \in S^c,
\end{align}
where we have denoted $N_x := N_{\varpi_x} = \dim(V_{\varpi_x})$, and $\{v_i\}_{i=1}^{N_{x}}$, and $\{f_i\}_{i=1}^{N_{x}}$, are the weight bases of $V_{\varpi_x}$, and $V_{-w_0(\varpi_x)}$ respectively, chosen in \textsection \ref{subsection:DJQG}.  We also find it convenient to introduce special notation for the following distinguished elements
\begin{align}
 z^{\varpi_x} := z^{\varpi_x}_{f_{N_x}}, & & \overline{z}^{\,\varpi_x} := \overline{z}_{v_{N_x}}^{\,\varpi_x}.
\end{align}
The elements $z^{\varpi_x}$ and $\overline{z}^{\,\varpi_x}$ have many distinguishing features which merit this separate notation, the most evident being that they are not contained in $\OO_q(G)^+$. 

\begin{eg} \label{eg:SU2Gens}
Note that $\OO_q(SU_2)$ is the quantum Poisson homogeneous space of the Podle\'s sphere $\OO_q(S^2)$, the unique quantum flag manifold of $\OO_q(SU_2)$. In this case there is just one fundamental representation $V_{\varpi_1}$ and it is of dimension $2$. The associated generators 
\begin{align*}
u_{12} := z^{\varpi_1}_{1} = c^{\varpi_1}_{12}, & &  u_{22} := z^{\varpi_1}_{2} = c^{\varpi_1}_{22}, & & u_{11} := \overline{z}^{\varpi_1}_{1} = c^{\varpi_1}_{11}, & & u_{21} :=  \overline{z}^{\varpi_1}_{2} = c^{\varpi_1}_{21},
\end{align*}
are just the matrix coefficients of $V_{\varpi_1}$, and so, they reduce to the well-known set of $\OO_q(SU_2)$ generators \cite[Proposition 2.13]{BeggsMajid:Leabh}.
\end{eg}


\subsection{Quantum Flag Manifolds and Principal Pairs} \label{subsection:QFMPP}

In this subsection we identify the Hopf algebras $K_{\OO_q(L_S)}$ and $\O_q(L_S^{\mathrm{s}})$ and conclude that the quantum Poisson homogeneous space and the quantum flag manifold give a principal pair. We begin with a lemma establishing centrality of the grouplike elements of $\OO_q(L_S)$.

\begin{lem} \label{lem:central}
Every element of $\Gamma_{\OO_q(L_S)}$ is central in $\OO_q(L_S)$.
\end{lem}
\begin{proof}
Note first that any $\gamma \in \Gamma_{\OO_q(L_S)}$ is a coordinate function associated to a one-dimensional representation $V_{\lambda}$, for $\lambda \in \mathcal{P}_S$. This means that any element of $U_q(\frak{l}_S^{\,\mathrm{s}})$ must act trivially on $V$, and hence the only generators of $U_q(\frak{l}_S)$ that act non-trivially on $V$ are $K_x^{\pm 1}$, for $x \in S^{\mathrm{c}}$. Consider a general monomial $X$ in the generators $E_j,F_j$, for $j \in S$, then take the general product
\begin{align*}
Y := (\Pi_{i=1}^r K_i^{a_i})X \in U_q(\frak{l}_S), & & \textrm{ for } a_1, \dots, a_r \in \mathbb{Z},
\end{align*}
and note that for all $h \in \OO_q(L_S)$
\begin{align*}
\gamma h (Y) = \gamma(Y_{(1)})h(Y_{(2)}) = \, & \gamma\big((\Pi_{i=1}^r K_i^{a_i})X_{(1)}\big)h((\Pi_{i=1}^r K_i^{a_i})X_{(2)}) \\
= & \, \gamma\big(\Pi_{x \in S^{\mathrm{c}}}K_x^{a_x}\big)h((\Pi_{i=1}^r K_i^{a_i})X).
\end{align*}
Similarly it holds that 
$$
h\gamma(Y) = h((\Pi_{i=1}^r K_i^{a_i})X)\gamma\big(\Pi_{x \in S^{\mathrm{c}}}K_x^{a_x}\big).
$$
Thus we see that $\gamma$ and $h$ commute for all $h \in \OO_q(L_S)$.
\end{proof}

\begin{prop} \label{thm:GAPPQFM}
The pair
$
\big(\OO_q(G/L_S), \, \OO_q(G/L^{\mathrm{s}}_S)\big)
$
is a group algebra principal pair.
\end{prop}
\begin{proof}
 Recall the Peter--Weyl decompositions 
\begin{align*}
\OO_q(L_S) \simeq \bigoplus_{\lambda + \mu \in \mathcal{P}_S^+ + \mathcal{P}^{\,}_{S^c}} C(U_{\lambda + \mu}), & & \OO_q(L^{\,\mathrm{s}}_S) = \bigoplus_{\mu \in \mathcal{P}_S^+} C(W_{\mu}),
\end{align*}
and consider the coarser decomposition of $\OO_q(L_S)$ given by 
\begin{align*}
\OO_q(L_S) \simeq \bigoplus_{\lambda \in \mathcal{P}_S} A_{\lambda}, & & \textrm{ where } A_{\lambda} := \bigoplus_{\mu \in \mathcal{P}^{{}^{\,}}_{S^c}} C(U_{\lambda + \mu}).
\end{align*}
We note that the kernel of the Hopf algebra surjection $\rho:\OO_q(L_S) \twoheadrightarrow \OO_q(L^{\mathrm{s}}_S)$ is homogeneous with respect to this coarser decomposition. It holds that  
\begin{align} \label{eqn:gens}
c^{V_{\lambda}}_{f,v} - 1 \in \ker(\rho), & & \textrm{ for all } \lambda  \in  \mathcal{P}^+_{S^c}, \, f \in U_{\lambda}^*, \, v \in U_{\lambda},
\end{align}
and moreover, the ideal $I$ generated by these elements is homogeneous with respect to the coarser grading, giving us a decomposition 
$$
I \simeq \bigoplus_{\lambda \in \mathcal{P}^+_S} I_{\lambda}.
$$
Now the quotient $A_{\lambda}/I_{\lambda}$ has dimension less than or equal to $C(W_{\lambda})$, which together with surjectivity of $\rho$ implies that the dimensions are equal, which is to say, the kernel of $\rho$ is generated by the elements given in \eqref{eqn:gens}. Thus we see that $\OO_q(L^{\,\mathrm{s}}_S)$ is isomorphic as a Hopf algebra to  $K_{\OO_q(L_S)}$. Thus, since $\OO_q(L_S)$ is cosemisimple by construction, and we have shown in the previous lemma that all the grouplike elements of $\OO_q(L_S)$ are central, the given pair is a principal pair.
\end{proof}

\begin{cor}
The categorical equivalence
\begin{align*}
{}^{~~~\,\OO_q(G)}_{\OO_q(G/L_S)}\mathrm{Mod}^{\OO_q(L^{\,\mathrm{s}}_S)} ~ \simeq ~ {}^{\OO_q(L_S)}\mathrm{Mod}^{{\OO_q(L^{\,\mathrm{s}}_S})}.
\end{align*}
follows immediately from Proposition \ref{prop:PPTak}.
\end{cor}

\begin{eg}
For any Drinfeld--Jimbo quantum group $U_q(\frak{g})$, the quantum flag manifold $\OO_q(F_G)$ associated to the choice of simple nodes $S = \varnothing$ is called the \emph{full quantum flag manifold} of $\OO_q(G)$. In this case the associated Hopf algebra $\OO_q(L_S)$ is isomorphic to $\OO(\mathbb{T}^{r})$, which is to say the group Hopf algebra of $\mathbb{Z}^{r}$. This gives a strong $\mathbb{Z}^{r}$-algebra grading
$$
\OO_q(G) \simeq \bigoplus_{\gamma \in \mathbb{Z}^{r}} \EE_{\gamma}.
$$
In other words, we have a decomposition of $\OO_q(G)$ into relative line modules. Since $\OO(\mathbb{T}^{r})$ is spanned by grouplike elements, this can be see as a degenerate case of Theorem \ref{thm:GAPPQFM} in the spirit of Example \ref{eg:degGAPP}.
\end{eg}

\begin{eg} \label{eg:SU2Grading}
The full quantum flag manifold of $\OO_q(SU_2)$ is the Podle\'s sphere, and the associated algebra grading of $\OO_q(SU_2)$ is a $\mathbb{Z}$-grading. Recalling the notation of Example \ref{eg:SU2Gens}, the grading is determined by 
\begin{align*}
\mathrm{deg}(u_{11}) = \mathrm{deg}(u_{21}) = - 1, & & \mathrm{deg}(u_{12}) = \mathrm{deg}(u_{22}) = 1.
\end{align*}
Thus we recover the well-known $\mathbb{Z}$-grading of $\OO_q(SU_2)$, as considered, for example, in \cite[\textsection 1]{Maj} up to a change of sign.
\end{eg}

\subsection{Relative Line Modules} \label{subsection:QFMlinebundles}

Since $\big(\OO_q(G/L_S), \, \OO_q(G/L^{\mathrm{s}}_S)\big)$
is a principal pair, Lemma \ref{lem:GAPP} implies the following result.

\begin{prop}
The direct sum decomposition of $\OO_q(G/L_S^{\,\mathrm{s}})$ into simple subobjects is a decomposition into relative line modules
$$
\OO_q(G/L_S^{\,\mathrm{s}}) \simeq \bigoplus_{\gamma \in \mathcal{P}_{S^c}} \EE_{\gamma}.
$$
This gives $\qphs$ the structure of a strongly $\mathcal{P}_{S^c}$-graded algebra. Moreover, every relative line module over $\OO_q(G/L_S)$ is isomorphic to $\EE_{\gamma}$, for some $\gamma \in \mathcal{P}_{S^{\mathrm{c}}}$. 
\end{prop}

We now give a concrete description of the decomposition of $\qphs$ into line bundles in terms of the generating set of $\OO_q(G/L^{\,\mathrm{s}})$ presented in  \eqref{eqn:QPHSGens}.

\begin{cor}
The $\mathcal{P}_{S^c}$-grading of $\OO_q(G/L^{\,\mathrm{s}}_S)$ is determined by 
\begin{align*}
\mathrm{deg}(z^{\varpi_x}_i) = \varpi_x, & & \mathrm{deg}(\overline{z}^{\,\varpi_x}_i) = -\varpi_x, 
\end{align*}
where $i = 1, \dots, N_x$, and $x \in S^c$.
\end{cor}
\begin{proof}
 The right $\OO_q(L_S)$-coaction on $\OO_q(G/L^{\,\mathrm{s}}_S)$ 
acts on $z^{\varpi_x}_i$ according to 
\begin{align*}
\Delta_{R,\OO_q(L_S)}(z^{\varpi_x}_i) = \sum_{a=1}^{N_x} c^{V_{\varpi_x}}_{v_i,f_a} \otimes \pi_{\OO_q(L_S)}(\,z^{\varpi_x}_{a}). 
\end{align*}
It follows from \eqref{eqn:GAPPIdentity} in Theorem \ref{thm:GAPP} that $\pi_{\OO_q(L_S)}(z^{\varpi_x})$ is an element of $\Gamma_{\OO_q(L_S)}$. 
Moreover, for any $y \in S^c$,
\begin{align*}
\langle  K_y, \pi_{\OO_q(L_S)}(\,z^{\varpi_x}_{a}) \rangle = \langle  K_y, z^{\varpi_x}_{a} \rangle = f_a(K_y \triangleright N_x) = q^{(\alpha_y,\varpi_x)} \delta_{a,N_x}, & & \textrm{ for all } a = 1, \dots, N_x.
\end{align*}
Thus we see that 
\begin{align*}
\Delta_{R,\OO_q(L_S)}(z^{\varpi_x}_i) \in \OO_q(G/L^{\,\mathrm{s}}_S) \otimes C(U_{\varpi_x}),
\end{align*}
which is to say, $z^{\varpi_x}_i$ is of degree $\varpi_x$. An analogous calculation shows that the degree of $\overline{z}^{\,\varpi_x}_i$ is $-\varpi_x$. Finally, since the $\mathcal{P}_{S^{\mathrm{c}}}$-grading is an algebra grading, it is completely determined by the degrees of the generators. 
\end{proof}


\subsection{A Principal $\ell$-Map}

We now produce an explicit principal $\ell$-map for our principal comodule algebra. This construction generalises the principal $\ell$-map constructed for the quantum Grassmannians in \cite[Corollary 4.13]{BWGrass}. For the special case of the Podle\'s sphere, the $\ell$-map reduces to the well-known $q$-monopole connection introduced in \cite{TBSM1, TBSMRevisited, HM99}.

In the statement of the proposition we find the following notation useful
\begin{align}
z_{\gamma} := \prod_{x \in S^{\mathrm{c}}} (\widecheck{z}^{\,\varpi_x})^{a_x}, & & \textrm{ for } \gamma = \sum_{x \in S^{\mathrm{c}}} a_x \varpi_x \in \mathcal{P}_{S^{\mathrm{c}}},
\end{align}
where we have denoted 
\begin{align*}
(\widecheck{z}^{\,\varpi_x})^{a_x} := (z^{\varpi_x})^{a_x}, ~~ \textrm{ if } a_x \geq 0, & & \textrm{ and } & & (\widecheck{z}^{\,\varpi_x})^{a_x} := (\overline{z}^{\,\varpi_x})^{a_x}, ~~ \textrm{ if }a_x < 0.
\end{align*}
Moreover, we find it convenient to denote 
\begin{align*}
t_{\gamma} := t_1^{a_1} \cdots \, t_{|S^{\mathrm{c}}|}^{a_{|S^{\mathrm{c}}|}}, & & \textrm{ for any } \gamma = \sum_{x \in S^{\mathrm{c}}} a_x \varpi_x \in \mathcal{P}_{S^{\mathrm{c}}}.
\end{align*}

\begin{prop}
For the Hopf--Galois extension $\OO_q(G/L_S) \hookrightarrow \OO_q(G/L_S^{\mathrm{s}})$, a principal $\ell$-map is given by 
\begin{align*}
\ell:\mathbb{C}\mathcal{P}_{S^c} \to \OO_q(G/L_S^{\mathrm{s}}) \otimes \OO_q(G/L_S^{\mathrm{s}}), & & t_{\gamma} \mapsto (S \otimes \id) \circ \Delta(z_{\gamma}).
\end{align*}
\end{prop}
\begin{proof} 
Let us first check that $\ell$ is well defined. Note first that since $\Delta$ is an algebra map, and $S$ is an anti-algebra map, it is enough to check that $\ell(z^{\varpi_x}) \in \O_q(G/L_S^{\mathrm{s}}) \otimes \O_q(G/L_S^{\mathrm{s}})$, for each $x \in \mathcal{P}_{S^c}$. Since 
\begin{align*}
 \ell(z^{\varpi_x}) = \sum_{a=1}^{N_x} S(z^{\varpi_x}_{f_{N_x},a}) \otimes z^{\varpi_x}_{a,f_{N_x}},
\end{align*}
this amounts to showing that $S(z^{\varpi_x}_{f_{N_x},a}) \in \OO_q(G/L_S^{\mathrm{s}})$, but this follows from the fact that 
\begin{align*}
X \triangleright S(z^{\varpi_x}_{f_{N_x},a}) = S(z^{\varpi_x}_{X \triangleright f_{N_x},a}) = \e(X) S(z^{\varpi_x}_{f_{N_x},a}), & & \textrm{ for all } X \in U_q(\frak{l}_S).
\end{align*}
Thus $\ell$ is indeed well defined.

By definition it holds that $\ell(1_{\mathbb{C}\mathcal{P}_{S^c}}) = 1_{\OO_q(G/L^{\mathrm{s}}_S)}$. Moreover, it is clear that $$
m_{\OO_q(G/L^{\mathrm{s}}_S)} \circ \ell = \e 1_{\OO_q(G/L^{\mathrm{s}}_S)}.
$$
Right $\mathbb{C}\mathcal{P}_{S^c}$-covariance of $\ell$ follows from the calculation 
\begin{align*}
(\id \otimes \Delta_R) \circ \ell(t_{\gamma}) = & \, (\id \otimes \Delta_R) \circ (S \otimes \id) \circ \Delta(z_{\gamma}) \\
= & \, (S \otimes \id \otimes \id) \circ (\Delta \otimes \id) \circ
(\Delta_R)(z_{\gamma})\\
= & \, (S \otimes \id \otimes \id) \circ (\Delta \otimes \id)(z_{\gamma} \otimes t_{\gamma})\\
= & \, ((S \otimes \id) \circ \Delta(z_{\gamma})) \otimes t_{\gamma}\\
= & \, \ell(t_{\gamma}) \otimes t_{\gamma}\\
= & \, (\ell \otimes \id) \circ \Delta(t_{\gamma}).
\end{align*}
Left $\mathbb{C}\mathcal{P}_{S^c}$-covariance of $\ell$ follows analogously, and so, we have a principal $\ell$-map.
\end{proof}


\begin{remark}
As explained in Appendix \ref{subsection:PrinConns}, associated to the principal $\ell$-map constructed above we have a principal connection $\Pi_{\ell}$ acting on the universal calculus of $\OO_q(G/L^{\,\mathrm{s}}_S)$. As is easily checked, the map $\Pi_{\ell}$ is a left $\OO_q(G)$-comodule map, meaning that every associated relative Hopf module connection $\nabla$ is also an  $\OO_q(G)$-comodule map. This extends the situation for the quantum Grassmannians discussed in \cite[Corollary 4.13]{BWGrass}.
\end{remark}

\subsection{Non-Cleftness of the Hopf--Galois Extension}

Let us recall the definition of a cleft comodule algebra, which can be viewed as a noncommutative generalisation of a trivial bundle: For $H$ a Hopf algebra, we say that a right $H$-comodule algebra $(P,\Delta_{R,H})$ is \emph{cleft} if there exists a convolution invertible right $H$-comodule map $j:H \to P$, which is necessarily injective. We call such a $j$ a \emph{cleaving map}. If $(P,\Delta_{R,H})$ is a Hopf--Galois extension of $B = P^{\co(H)}$, then the existence of a cleaving map is equivalent to $A$ having the \emph{normal basis property}, which is to say, equivalent to the existence of a left $B$-module and right $H$-comodule isomorphism $A \simeq B \otimes H$, see \cite[Proposition 2.3]{B99} for details.

For the special case of the Podle\'s sphere, it was shown in \cite{HM99} that the Hopf--Galois extension $\OO_q(S^2) \hookrightarrow \OO_q(SU_2)$ is not cleft.  We now extend this result to include all quantum flag manifolds. 

\begin{prop}
The Hopf--Galois extension $\OO_q(G/L_S) \hookrightarrow \OO_q(G/L_S^{\,\mathrm{s}})$ is non-cleft.
\end{prop}
\begin{proof}
Let us assume that we have a convolution invertible right  $\mathbb{C}\mathcal{P}_S$-comodule map 
$$
j:\mathbb{C}\mathcal{P}_S \to \O_q(G/L^{\,s}_S).
$$
Choose an element $x \in S^c$. The fact that $j$ is convolution invertible, and that $t_{x}$ is grouplike, means that $j(t_{x})$ is invertible. Moreover, since $j$ is a $\mathbb{C}\mathcal{P}_S$-comodule map, $j(t_{x})$ must have degree $\varpi_x$.

Consider the associated inclusion $i_x:U_q(\frak{sl}_2) \hookrightarrow U_q(\frak{g})$,  
along with the dual Hopf algebra surjection $\pi_x:\OO_q(G) \to \OO_q(SU_2)$. Taking $K \in U_q(\frak{sl}_2)$, we see that 
\begin{align*}
K \triangleright \pi_x(j(t_x)) 
=  \, \pi_x(j(t_x)_{(1)})  \langle K_x, j(t_x)_{(2)} \rangle 
=  \, q^{(\alpha_x,\alpha_x)/2}\pi_x(j(t_x)).
\end{align*}
Thus, with respect to the $\mathbb{Z}$-grading on $\OO_q(SU_2)$ induced by the action of the generators $K$ and $K^{-1}$, the element $\pi_x(j(t_{x}))$ has non-zero degree.

Since $\pi_x$ is an algebra map, the element $\pi_x(j(t_x))$ is necessarily invertible in $\OO_q(SU_2)$. It was established in \cite[Appendix]{HM99} that the only invertible elements in $\OO_q(SU_2)$ are scalar multiples of the identity. However, this contradicts the fact that $\pi_x(j(t_x))$ has non-zero degree. Thus we conclude that no such $j$ exists, which is to say, that the Hopf--Galois extension is non-cleft.
\end{proof}

\subsection{Schneider's Equivalence and a Family of Examples} \label{subsection:IrredQFM}

Since $\qphs$ and $\OO_q(G/L_S)$ form a principal pair, the equivalence
\begin{align*}
^{~~~~\OO_q(G)}_{\OO_q(G/L^{\mathrm{s}}_S)}\mathrm{Mod}^{\mathrm{gr}(\mathcal{P}_{S^{\mathrm{c}}})} ~ \simeq ~ {}^{\OO_q(L_S)}\mathrm{Mod}
\end{align*}
follows immediately from \eqref{eqn:ASchneiderGAPP}. In this subsection, we present a naturally occurring family of objects, in the left hand side category, coming from quantum group noncommutative geometry.

A quantum flag manifold $\OO_q(G/L_S)$ is said to be \emph{irreducible} if $S^{\mathrm{c}}$ contains a single root vector $\alpha_x$ and this root vector appears in any positive root of $\frak{g}$ with coefficient at most one. In \cite{HK,HKdR} Heckenberger and Kolb introduced a left $\OO_q(G)$-covariant first-order differential calculus $\Omega^1_q(G/L_S)$ for each irreducible quantum flag manifold $\OO_q(G/L_S)$. These calculi are of classical dimension and $q$-deform the classical space of $1$-forms over $G/L_S$. These celebrated calculi are objects of central importance for the noncommutative geometry of quantum groups, and are naturally objects in the category of relative Hopf modules $^{~~~~\OO_q(G)}_{\OO_q(G/L_S)}\mathrm{Mod}$.

For the irreducible case, the weight sublattice $\mathcal{P}_{S^{\mathrm{c}}}$ is isomorphic to $\mathbb{Z}$, and so, the category ${}^{~~~~\OO_q(G)}_{\OO_q(G/L^{\mathrm{s}}_S)}\mathrm{Mod}^{\mathrm{gr}(\mathcal{P}_{S^{\mathrm{c}}})}$ specialises to the category
$$
{}^{~~~~\OO_q(G)}_{\OO_q(G/L^{\mathrm{s}}_S)}\mathrm{Mod}^{\mathrm{gr}(\mathbb{Z})}.
$$
In \cite[\textsection 3.2]{HKdR} Heckenberger and Kolb constructed left $\OO_q(G)$-covariant first-order calculi $\Omega^1_q(G/L^{\,\mathrm{s}}_S)$ for the associated quantum Poisson homogeneous spaces $\OO_q(G/L^{\,\mathrm{s}}_S)$. These calculi are naturally objects in the 
category ${}^{~~~~\OO_q(G)}_{\OO_q(G/L^{\mathrm{s}}_S)}\mathrm{Mod}$. Moreover, as shown in \cite{CDOBBW}, each calculus is homogeneous with respect to the $\mathbb{Z}$-grading on $\OO_q(G/L^{\,\mathrm{s}}_S)$, and so,
$$
\Omega^1_q(G/L^{\,\mathrm{s}}_S) \in {}^{~~~~\OO_q(G)}_{\OO_q(G/L^{\mathrm{s}}_S)}\mathrm{Mod}^{\mathrm{gr}(\mathbb{Z})}.
$$
It then follows from \cite[Proposition 3.10]{BWGrass} that $\Omega^1_q(G/L^{\,\mathrm{s}}_S)$ gives a quantum principal bundle over the base $\OO_q(G/L_S)$. This is a key step in the proof of the Borel--Weil theorem for the irreducible quantum flag manifolds given in \cite{CDOBBW}.


\section{Examples of Non-Principal Noncommutative Fibrations}

In this section we construct two families of non-principal noncommutative fibrations, and realise many well-known examples of quantum homogeneous spaces as noncommutative fibers.

\subsection{Non-Principal Examples with Irreducible Quantum Flag Manifold Bases} 

For any Drinfeld--Jimbo quantum group $U_q(\frak{g})$, and a pair of subsets $S_P \sseq S_B\subseteq \Pi$, we have the triple of algebras 
$$
\OO_q(G/L_{S_B}) \hookrightarrow \OO_q(G/L_{S_P}) \twoheadrightarrow \OO_q(L_{S_P}/L_{S_B}),
$$
where we have denoted 
$$
\OO_q(L_{S_B}/L_{S_P}) := \OO_q(L_{S_B})^{\co\left(\OO_q(L_{S_P})\right)}.
$$
As discussed in Remark \ref{remark:noncommutativefibration}, we think of this triple as a type of noncommutative homogeneous fibration. For special cases we get fibrations built from well-known examples of quantum homogeneous spaces. Here we present a series of special cases where the base of the fibration is an irreducible quantum flag manifold, as discussed in \textsection \ref{subsection:IrredQFM}. Throughout we use Humphrey’s numbering of the Dynkin nodes \cite[\textsection 11.4]{Humph} of $\frak{g}$.

Note that for any two subsets $S_B \subseteq S_P \subseteq \Pi$, we have an obvious algebra isomorphism 
$$
\OO_q(L_{S_B}/L_{S_P}) \simeq \OO_q(L^{\mathrm{s}}_{S_B})^{\co\left(\OO_q(L_{S_P \cap S_B})\right)}.
$$
In fact, this is an isomorphism of $\OO_q(L^{\mathrm{\,s}}_{S_B})$-comodules. A motivating example here is the presentation of $\OO_q(\mathbb{CP}^{n})$ as a quantum homogeneous  $\OO_q(U_{n+1})$-space. We will tacitly use this isomorphism in the examples that follow.

\begin{eg}
We begin with the example introduced by Brzezi\'{n}ski and Szyma\'{n}ski \cite{MKSzy20} as a motivating example for a proposed theory of \emph{noncommutative fiber bundles} \cite{TBWS}. For the Drinfeld--Jimbo quantum group $U_q(\frak{sl}_3)$, we choose $S^c_B = \{\alpha_2\}$ and $S^c_P = \{\alpha_1, \alpha_2\},$ which we represent respectively with the coloured Dynkin diagrams
\begin{center}
\begin{tikzpicture}[scale=.5]
\draw
(0,0) circle [radius=.25];

\draw[fill=black]
(2,0) circle [radius=.25];

\draw[thick]
(.25,0) -- (1.75,0);
\end{tikzpicture} \, ~~~~~ and ~~~~~ \, \begin{tikzpicture}[scale=.5]
\draw
;

\draw[fill=black]
(0,0) circle [radius=.25]
(2,0) circle [radius=.25];

\draw[thick]
(.25,0) -- (1.75,0);
\end{tikzpicture}.
\end{center}
This gives us the noncommutative fibration 
$$
\OO_q(\mathbb{CP}^2) \hookrightarrow \OO_q(F_{SU_3}) \twoheadrightarrow \OO_q(S^2), 
$$
where $\OO_q(\mathbb{CP}^2)$ is the quantum projective plane, and $\OO_q(S^2) = \OO_q(\mathbb{CP}^1)$ is the standard Podle\'s sphere.
\end{eg}

We now generalise this construction to higher orders. There are two natural ways to do this. First we consider the case where the fiber and the total space are higher order full quantum flag manifolds, where we note that the Podle\'s sphere is in fact the full quantum flag manifold of $\OO_q(SU_2)$.

\begin{eg}
For the Drinfeld--Jimbo quantum group $U_q(\frak{sl}_{n+1})$, we choose $S^c_B = \{\alpha_{n}\}$ and $S^c_P = \Pi$ which we represent  graphically with the respective coloured Dynkin diagrams
\begin{center}
\begin{tikzpicture}[scale=.5]
\draw
(0,0) circle [radius=.25]
(6,0) circle [radius=.25] 
(2,0) circle [radius=.25] 
(4,0) circle [radius=.25] ; 

\draw[fill=black]
(8,0) circle [radius=.25];

\draw[thick,dotted]
(4.25,0) -- (5.75,0);
\draw[thick]
(.25,0) -- (1.75,0)
(2.25,0) -- (3.75,0)
(6.25,0) -- (7.75,0);
\end{tikzpicture} \, ~~~~~ and ~~~~~ \, \begin{tikzpicture}[scale=.5]


\draw[fill=black]
(0,0) circle [radius=.25]
(2,0) circle [radius=.25] 
(4,0) circle [radius=.25] 
(6,0) circle [radius=.25]
(8,0) circle [radius=.25];

\draw[thick,dotted]
(4.25,0) -- (5.75,0);
\draw[thick]
(.25,0) -- (1.75,0)
(2.25,0) -- (3.75,0)
(6.25,0) -- (7.75,0);
\end{tikzpicture}.
\end{center}
This gives us the fibration 
$$
\OO_q(\mathbb{CP}^n) \hookrightarrow \OO_q(F_{SU_{n+1}}) \twoheadrightarrow \OO_q(F_{SU_{n}}),
$$
where $\OO_q(\mathbb{CP}^n)$ is quantum projective $n$-space.
\end{eg}

We now consider an alternative generalisation of Brzezi\'{n}ski and Szyma\'{n}ski's fibration, where both the base and the fiber are higher order quantum projective spaces. 

\begin{eg}
For the Drinfeld--Jimbo quantum group $U_q(\frak{sl}_{n+1})$, we choose $S^c_B = \{\alpha_{n}\}$ and $S^c_P = \{\alpha_1, \alpha_{n}\},$ which we represent  graphically with the respective coloured Dynkin diagrams
\begin{center}
\begin{tikzpicture}[scale=.5]
\draw
(0,0) circle [radius=.25]
(6,0) circle [radius=.25] 
(2,0) circle [radius=.25] 
(4,0) circle [radius=.25] ; 

\draw[fill=black]
(8,0) circle [radius=.25];

\draw[thick,dotted]
(4.25,0) -- (5.75,0);
\draw[thick]
(.25,0) -- (1.75,0)
(2.25,0) -- (3.75,0)
(6.25,0) -- (7.75,0);
\end{tikzpicture} \, ~~~~~ and ~~~~~ \, \begin{tikzpicture}[scale=.5]
\draw
(2,0) circle [radius=.25] 
(4,0) circle [radius=.25] 
(6,0) circle [radius=.25] ; 
\draw[fill=black]
(0,0) circle [radius=.25]
(8,0) circle [radius=.25];

\draw[thick,dotted]
(4.25,0) -- (5.75,0);
\draw[thick]
(.25,0) -- (1.75,0)
(2.25,0) -- (3.75,0)
(6.25,0) -- (7.75,0);
\end{tikzpicture}.
\end{center}
This gives us the fibration 
$$
\OO_q(\mathbb{CP}^{n}) \hookrightarrow \OO_q\big(SU_{n+1}/(U_{n-1} \times U_1)\big) \twoheadrightarrow \OO_q(\mathbb{CP}^{n-1}),
$$ 
where we note that for $n=2$, the quantum flag manifold $\OO_q\big(SU_{n+1}/U_{n-1} \times U_1\big)$ reduces to the full quantum flag manifold of $\OO_q(SU_3)$.
\end{eg}

There is no need to confine our attention to the $A$-series, so let us consider a $C$-series example. 

\begin{eg}
For the Drinfeld--Jimbo quantum group $U_q(\frak{sp}_{2n})$, we choose $S^c_B = \{\alpha_{n}\}$ and $S^c_P = \Pi$ which we represent graphically  with the respective coloured Dynkin diagrams
\begin{center}
\begin{tikzpicture}[scale=.5]
\draw
(0,0) circle [radius=.25] 
(2,0) circle [radius=.25] 
(4,0) circle [radius=.25] 
(6,0) circle [radius=.25] ; 
\draw[fill=black]
(8,0) circle [radius=.25];

\draw[thick]
(.25,0) -- (1.75,0);

\draw[thick,dotted]
(2.25,0) -- (3.75,0)
(4.25,0) -- (5.75,0);

\draw[thick] 
(6.25,-.06) --++ (1.5,0)
(6.25,+.06) --++ (1.5,0); 

\draw[thick]
(7,0) --++ (60:.2)
(7,0) --++ (-60:.2);
\end{tikzpicture}\, ~~~~~ and ~~~~~ \, \begin{tikzpicture}[scale=.5]

\draw[fill=black]
(0,0) circle [radius=.25] 
(2,0) circle [radius=.25] 
(4,0) circle [radius=.25] 
(6,0) circle [radius=.25]
(8,0) circle [radius=.25];

\draw[thick]
(.25,0) -- (1.75,0);

\draw[thick,dotted]
(2.25,0) -- (3.75,0)
(4.25,0) -- (5.75,0);

\draw[thick] 
(6.25,-.06) --++ (1.5,0)
(6.25,+.06) --++ (1.5,0); 

\draw[thick]
(7,0) --++ (60:.2)
(7,0) --++ (-60:.2);
\end{tikzpicture}.
\end{center}
This gives us the noncommutative fibration 
$$
\OO_q(\mathbf{L}_{n}) \hookrightarrow \OO_q(F_{Sp_{n}}) \to \OO_q(F_{SU_n}),
$$
where $\OO_q(\mathbf{L}_{n})$ is the quantum Lagrangian Grassmannian. In the commutative case, $\mathbf{L}_n$ is the space of complex structures on $\mathbb{H}^n$ compatible with the canonical inner product. 
\end{eg}


Let us next consider an example from the exceptional Drinfeld--Jimbo quantum groups.

\begin{eg}
For the Drinfeld--Jimbo quantum group $U_q(\frak{e}_6)$, we choose $S^c_B = \{\alpha_{6}\}$ and $S^c_P = \{\alpha_5,\, \alpha_6\},$ which we represent graphically with the respective coloured Dynkin diagrams
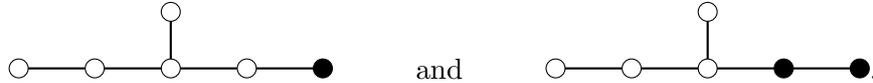
\begin{center}
\begin{tikzpicture}[scale=.5]
\draw
(0,0) circle [radius=.25] 
(2,0) circle [radius=.25] 
(4,0) circle [radius=.25] 
(4,1.5) circle [radius=.25]
(6,0) circle [radius=.25] ;

\draw[fill=black] 
(8,0) circle [radius=.25];
\draw[thick]
(.25,0) -- (1.75,0)
(2.25,0) -- (3.75,0)
(4.25,0) -- (5.75,0)
(6.25,0) -- (7.75,0)
(4,.25) -- (4, 1.25);
\end{tikzpicture} \, ~~~~~ and ~~~~~ \, \begin{tikzpicture}[scale=.5]
\draw
(0,0) circle [radius=.25] 
(2,0) circle [radius=.25] 
(4,0) circle [radius=.25] 
(4,1.5) circle [radius=.25];

\draw[fill=black] 
(6,0) circle [radius=.25] 
(8,0) circle [radius=.25];
\draw[thick]
(.25,0) -- (1.75,0)
(2.25,0) -- (3.75,0)
(4.25,0) -- (5.75,0)
(6.25,0) -- (7.75,0)
(4,.25) -- (4, 1.25);
\end{tikzpicture}.
\end{center}
This gives us the quantum homogeneous fibration 
$$
\OO_q(\mathbb{OP}^2) \hookrightarrow \OO_q(E_6/(SO_{10} \times U_1)) \twoheadrightarrow \OO_q(\mathbf{S}_5),
$$
where $\OO_q(\mathbb{OP}^2)$ is the quantum Caley plane, that is, the irreducible quantum flag manifold of $\OO_q(E_6)$, and $\OO_q(\mathbf{S}_5)$ is the quantum spinor variety, one of the two $D_5$ irreducible quantum flag manifolds. 
In the classical case $\mathbf{S}_5$ is the space of orthogonal complex structures on $\mathbb{R}^{10}$. 
\end{eg}


%

\subsection{Noncommutative Sphere Fibrations and Quantum Stiefel Manifolds} 

In this subsection we generalise a different aspect of Brzezi\'{n}ski and Szyma\'{n}ski's noncommutative fibration, producing fibrations with higher order quantum spheres as \emph{noncommutative fibers}. Note first that, for any complex semisimple Lie algebra $\frak{g}$ and any two subsets $S_P \sseq S_B \subseteq \Pi$, a quantum homogeneous fibration is given by the triple
$$
\OO_q(G/L_{S_B}^{\mathrm{s}}) \hookrightarrow \OO_q(G/L^{\mathrm{s}}_{S_P}) \twoheadrightarrow \OO_q(L^{\mathrm{s}}_{S_B}/L^{\mathrm{s}}_{S_P}),
$$
where we have denoted by $\OO_q(L^{\mathrm{s}}_{S_B}/L^{\mathrm{s}}_{S_P})$ the space of $\OO_q(L^{\,\mathrm{s}}_{S_P})$-coinvariants of $\OO_q(L^{\mathrm{s}}_{S_B})$. As we will see, for special cases we get fibrations built from well-known examples of quantum homogeneous spaces.

With respect to Humphrey's numbering of the Dynkin nodes \cite[\textsection 11.4]{Humph}, for any non-exceptional simple Lie algebra $\frak{g}$, we denote
\begin{align*}
S_{> m} := \{\alpha_{m+1}, \dots, \alpha_{\mathrm{n}}\}.
\end{align*}
Associated to the Hopf subalgebra $U_q(\frak{l}^{\mathrm{\,s}}_{S_{>m}})$ and its quantum subgroup $\OO_q(L_{S_{>m}})$, we have the quantum homogeneous $\OO_q(G)$-space of 
$\OO_q(L_{S_{>m}})$-coinvariants, which we call the \emph{quantum $m$-Stiefel manifold} of $\OO_q(G)$.

\begin{eg}
It is instructive to observe that each Stiefel manifold is the total space of a principle comodule algebra. Indeed, by the results of \textsection \ref{subsection:QFMPP}, the quantum Stiefel manifold associated to any $S_{>m}$ is a principal $\mathbb{C}\mathcal{P}_{S^c_{>m}}$-comodule algebra over the quantum flag manifold $\OO_q(G/L_{S_{>m}})$. 
\end{eg}

Let us now place some examples of quantum Stiefel manifolds into  $q$-deformed non-principal fibrations.

\begin{eg}
Consider the Drinfeld--Jimbo quantum group $U_q(\frak{sl}_{n+1})$, for $n \geq 2$, and the subset $S_{>3}$ of simple nodes, for $m = 1, \dots, n$. In this case the quantum Stiefel manifolds reduce to the quantum complex Stiefel manifolds $\OO_q(V_p\mathbb{C}^{n+1})$ introduced by Podkolzin and Vainerman \cite{PodV}, and further studied in \cite{Bip20}. Note that for the degenerate case of $m=1$, we get back the quantum spheres $\OO_q(S^{2n-1})$ of Vaksman and Soibelman \cite{VaksS}.

Choosing $S_B = S_{>1}$ and $S_P = S_{>2}$, which we represent graphically with the respective coloured Dynkin diagrams
\begin{center}
\begin{tikzpicture}[scale=.5]
\draw
(2,0) circle [radius=.25]
(4,0) circle [radius=.25] 
(6,0) circle [radius=.25] 
(8,0) circle [radius=.25];

\draw[fill=black]
(0,0) circle [radius=.25];

\draw[thick,dotted]
(4.25,0) -- (5.75,0);
\draw[thick]
(.25,0) -- (1.75,0)
(2.25,0) -- (3.75,0)
(6.25,0) -- (7.75,0);
\end{tikzpicture} \, ~~~~~ and ~~~~~ \, \begin{tikzpicture}[scale=.5]


\draw[fill=black]
(0,0) circle [radius=.25]
(2,0) circle [radius=.25];
\draw
(4,0) circle [radius=.25] 
(6,0) circle [radius=.25]
(8,0) circle [radius=.25];

\draw[thick,dotted]
(4.25,0) -- (5.75,0);
\draw[thick]
(.25,0) -- (1.75,0)
(2.25,0) -- (3.75,0)
(6.25,0) -- (7.75,0);
\end{tikzpicture},
\end{center}
we get the quantum homogeneous fibration 
$$
\OO_q(S^{2n+1}) \hookrightarrow \OO_q(V_2\mathbb{C}^{n+1}) \twoheadrightarrow \OO_q(S^{2n-1}).
$$
Note that for the special case of $n=3$, the fiber reduces to $\OO_q(S^3) = \OO_q(SU_2)$, so the fibration is actually principal.
\end{eg}


\begin{eg}
For the Drinfeld--Jimbo quantum group $U_q(\frak{sp}_{2n})$, we choose $S_B = S_{> 1}$ and $S_P = S_{> 3},$ which we represent  graphically  with the respective coloured Dynkin diagrams
\begin{center}
\begin{tikzpicture}[scale=.5]
\draw[fill=black]
(0,0) circle [radius=.25]; 
\draw
(2,0) circle [radius=.25] 
(4,0) circle [radius=.25] 
(6,0) circle [radius=.25] 
(8,0) circle [radius=.25];

\draw[thick]
(.25,0) -- (1.75,0);

\draw[thick,dotted]
(2.25,0) -- (3.75,0)
(4.25,0) -- (5.75,0);

\draw[thick] 
(6.25,-.06) --++ (1.5,0)
(6.25,+.06) --++ (1.5,0); 

\draw[thick]
(7,0) --++ (60:.2)
(7,0) --++ (-60:.2);
\end{tikzpicture}\, ~~~~~ and ~~~~~ \, \begin{tikzpicture}[scale=.5]
\draw[fill=black]
(0,0) circle [radius=.25] 
(2,0) circle [radius=.25];

\draw
(4,0) circle [radius=.25] 
(6,0) circle [radius=.25]
(8,0) circle [radius=.25];

\draw[thick]
(.25,0) -- (1.75,0);

\draw[thick,dotted]
(2.25,0) -- (3.75,0)
(4.25,0) -- (5.75,0);

\draw[thick] 
(6.25,-.06) --++ (1.5,0)
(6.25,+.06) --++ (1.5,0); 

\draw[thick]
(7,0) --++ (60:.2)
(7,0) --++ (-60:.2);
\end{tikzpicture}.
\end{center}

The Stiefel manifold corresponding to $S_{> 1}$ is a $q$-deformation of the $C$-series presentation of the sphere $S^{4n-1}$ as the homogeneous sphere 
$$
S^{4n-1} \simeq Sp_{2n}/Sp_{2(n-1)}.
$$
To differentiate it from the $A$-series quantum sphere, we denote it by $\OO_q(S^{4n-1}_{\mathbb{H}})$ and call it the \emph{quantum quaternionic sphere}. The quantum quaternionic sphere has appeared a number of times in the literature. It was studied in \cite{GPR05} and \cite{BS19FF, TBWS,MKSzy20} as the total space of a noncommutative instanton fibration, where it was called the \emph{quantum symplectic sphere}. Moreover, the representations and $K$-theory of its $C^*$-algebra completion were examined in \cite{Sau2015}.

The Stiefel manifold corresponding to $S_{> 2}$ is a $q$-deformation of the coordinate algebra of the classical quaternionic Stiefel manifold $V_2\mathbb{H}^n$. Hence we call it the \emph{quantum quaternionic $2$-plane Stiefel manifold} and denote it by $\O_q(V_2\mathbb{H}^n)$. 

Our choice of $S_B$ and $S_P$ gives the quantum homogeneous fibration 
$$
\OO_q(S^{4n-1}_{\mathbb{H}}) \hookrightarrow \OO_q(V_2\mathbb{H}^n) \twoheadrightarrow \OO_q(S_{\mathbb{H}}^{4n-5}).
$$
Note that for the degenerate case of $n=2$, the quantum sphere $\OO_q(S_{\mathbb{H}}^{4n-5})$ reduces to the Hopf algebra $\OO_q(S^3) = \OO_q(SU_2)$, and so, the fibration is principal.
\end{eg}

\appendix

\section{Some Comments on Quantum Homogeneous Spaces}

In this paper we use a definition of a quantum homogeneous space that is different (although equivalent) to the one used in our previous works. Thus we spend some time discussing the various formulations of the definition. Motivated by the important role relative line modules play in the paper, we also discuss some results about simple relative Hopf modules.

\subsection{Quantum Homogeneous Spaces as Principal Comodule Algebras}

 For any quantum homogeneous space $B = A^{\co(\pi_B(A))}$, it is well-known that $A$ is automatically a Hopf--Galois extension of $B$, see for example \cite[Lemma 3.9]{T72}. Thus it follows from \cite[Theorem 1]{MulSch} that $A$ is also faithfully flat as a left $B$-module, and hence that it is a principal comodule algebra.

Cosemisimplicity allows us to verify the faithful flatness condition for a general coideal subalgebra $B \subseteq A$ satisfying $B^+A = AB^+$. Explicitly, if the quotient Hopf algebra  $A/B^+A$ is cosemisimple, then $A$ is automatically faithfully flat as a right $B$-module, see\cite[\textsection 3.3]{DOKSS} for a detailed discussion.

\subsection{Quantum Homogeneous Spaces, Ideals, and Quantum Subgroups}\label{app:QHSIdealsQSGs}

In the literature it is common to consider a weaker version of the quantum homogeneous space definition: coideal subalgebras $B \sseq A$ for which $A$ is faithfully flat as a right $B$-module, but for which $AB^+$ is not necessarily equal to $B^+A$. As shown in \cite[Theorem 1.11]{Masuoka}, these \emph{generalised quantum homogeneous spaces} are in bijective correspondence with two-sided coideals $I \subseteq A$ that are also right $A$-ideals and for which $A$ is faithfully coflat as a right $A/I$-comodule. Explicitly, the correspondence is given by associating a generalised quantum homogeneous space $B$ to the right ideal two-sided coideal $B^+A \sseq A$, and associating to a right ideal two-sided coideal $I$ the space of coinvariants associated to the coalgebra surjection $A \to A/I$. 

Let us now specialise to quantum homogeneous spaces. In this case the $B^+A$ will be a two-sided ideal and two-sided coideal. In fact, by Koppinen's lemma (see \cite[Lemma 1.4]{MulSch}) it will be a Hopf ideal, meaning that the quotient $A/I$ will be a Hopf algebra. To go in the other direction, we need the following technical lemma.

\begin{lem} \label{prop:Aplus} For any surjective Hopf map $\pi:A \to H$, the associated space of coinvariants $B = A^{\co(H)}$ satisfies $B^+A = AB^+$. 
\end{lem}
\begin{proof}
For any $a \in A$ and $b \in B^+$,  it holds that
$
ab = a_{(1)}bS(a_{(2)})a_{(3)},
$
and that $\e(a_{(1)}bS(a_{(2)})) = 0$.  Moreover, we see that 
\begin{align*}
\Delta_{R,H}(a_{(1)}bS(a_{(2)})) = & \, a_{(1)}b_{(1)}S(a_{(4)}) \otimes \pi(a_{(2)})\pi(b_{(2)})\pi(S(a_{(3)})) \\
= & \, a_{(1)}bS(a_{(4)}) \otimes \pi(a_{(2)})\pi(S(a_{(3)}))\\
= & \, a_{(1)}bS(a_{(2)}) \otimes 1.
\end{align*}
From this we can conclude that
$
ab = a_{(1)}bS(a_{(2)})a_{(3)} \in B^+A,
$
implying the inclusion $AB^+ \sseq B^+A$. The opposite inclusion is established analogously, giving the required equality.
\end{proof}

Thus we have a bijection between the quantum homogeneous $A$-spaces and the set 
$$
\Big\{I \subseteq A ~|~ I \textrm{ a Hopf ideal such that $A$ is faithfully coflat as a right $A/I$-comodule} \Big\}\!.
$$

We finish with a second equivalent characterisation of quantum homogeneous spaces. For any Hopf algebra $A$, a \emph{quantum $A$-subgroup} is a pair $(H,\pi)$ consisting of a Hopf algebra $H$ and and a surjective Hopf algebra map $\pi:A \to H$, such that $A$ is faithfully coflat as a right $A/\ker(\pi)$-comodule. Define an equivalence relation on quantum $A$-subgroups by setting $(H,\pi)$ equivalent to $(H',\pi')$ if $\ker(\pi) = \ker(\pi')$. Note that two quantum subgroups $(H,\pi)$ and $(H',\pi')$ are equivalent if and only if there exists a Hopf algebra isomorphism $\phi: H \to H'$ such that $\phi \circ \pi = \pi'$. This gives a bijective correspondence between quantum homogeneous spaces and the set 
$$
\Big\{ [(H,\pi)] ~|~ (H,\pi) \textrm{ is a quantum $A$-subgroup} \Big\},
$$
where $[(H,\pi)]$ denotes the equivalence class of $(H,\pi)$.


\subsection{Simple Objects of the Category ${}^{A}_{B}\mathrm{Mod}_0$} \label{app:simplemod0}

Since relative line modules are a primary object of interest in this paper, we find it appropriate to make some general comments about simple objects in the category ${}^A_{B}\mathrm{Mod}_0$, 
where as usual $A$ is a Hopf algebra, and $B$ is a quantum homogeneous $A$-space 

Note first that since $ab = a_{(1)}bS(a_{(2)})a_{(3)}$, for all $a \in A$, $b \in B$, the Hopf algebra $A$ (endowed with its obvious right $B$-module structure) is an object in ${}^A_B\mathrm{Mod}_0$. Moreover, any coideal sub-$B$-bimodule of $A$ will also be an object in ${}^A_B\mathrm{Mod}_0$. 

Conversely, consider a simple object $\F \in {}^A_B\mathrm{Mod}_0$. Since $\Phi(\F)$ is a simple $\pi_P(A)$-comodule, it admits an embedding into $\pi_P(A)$, implying that $\F$ embeds into $A$. Thus we see that any simple object in ${}^A_B\mathrm{Mod}_0$ can be embedded as a sub-object of $A$. In particular, all relative line modules embed into $A$. Note that, with respect to an embedding $\F \hookrightarrow A$, it holds that $
\F \simeq A \square_{\pi_B} \pi_B(\F).
$

Consider now the subcategory ${}^A_B\mathrm{mod}_0 \subseteq {}^A_B\mathrm{Mod}_0$ of finite-dimensional left $H$-comodules, and note that it is a rigid monoidal category. From the monoidal form of Takeuchi's equivalence, we see that the corresponding subcategory ${}^A_B\mathrm{mod}_0 \subseteq {}^A_B\mathrm{Mod}_0$ is also a rigid monoidal category. Thus every $\F \in {}^{A}_B\mathrm{mod}_0$ admits a left and a right dual, implying that $\F$ is finitely generated and projective as a left and as a right $B$-module, and that ${}^A_B\mathrm{mod}_0$ can be described as the subcategory whose objects are finitely generated left $B$-modules. Moreover, any object in ${}^A_B\mathrm{Mod}_0$ which is a direct sum of objects in ${}^A_B\mathrm{mod}_0$ will also be projective as a left and as a right $B$-module. 

Finally, we consider the case where $\pi_B(A)$ is a cosemisimple Hopf algebra. The Peter--Weyl decomposition $\pi_B(A) \simeq \bigoplus_{\alpha \in \widehat{\pi_B}} \pi_B(A)_{\alpha}$ gives a corresponding decomposition of $A$ into simple sub-objects
\begin{align*}
A \simeq \bigoplus_{\alpha \in \widehat{\pi_B}} A_{\alpha},    
\end{align*}
where we have denoted 
$$
A_{\alpha} := \left\{a \in A \, |\, \pi_B(a) \in \pi_B(A)_{\alpha}\right\} = \{ a \in A \,|\, a_{(1)} \otimes \pi_B(a_{(2)}) \in A \otimes \pi_B(A)_{\alpha}\}.
$$
It is instructive to note that
$$
\Phi(A_{\alpha}) = \pi_B(A_{\alpha}) = \pi_B(A)_{\alpha}.
$$
Cosemisimlicity of $\pi_B(A)$ implies that $A$ is isomorphic to the direct sum of the simple objects in ${}^A_B\mathrm{mod}_0$. Moreover, since each simple object is projective as a left and as a right $B$-module, each $\F \in {}^{A}_B\mathrm{Mod}_0$ must be projective as a left and as a right $B$-module. Thus we have recovered some well-known results about relative Hopf modules (see for example \cite[Corollary 2.4]{DGH01} and \cite[Corollary 1.5]{MulSch}) from the monoidal form of Takeuchi's equivalence.

\subsection{Connections and Projectivity} \label{subsection:PrinConns}

Principal $\ell$-maps are the link between principal comodule algebras and noncommutative geometry. In particular, principal $\ell$-maps are in bijective correspondence with principal connections $\Pi:\Omega^1_u(P) \to \Omega^1_u(P)$, where 
$$
\Omega^1_u(P) = \ker(m:P \otimes P \to P)
$$ 
is the universal first-order differential calculus 
over $P$. Principal connections in turn provide a systematic procedure for constructing connections $\nabla:\F \to \Omega^1_u(P) \otimes_P \F$, where $\F$ is an \emph{associated module}, which is to say a module of the form $ P \square_H V_{\F}$, for some left $H$-comodule $V_{\F}$. See the recent monograph \cite[Chapter 5]{BeggsMajid:Leabh}, or the associated papers \cite{CDOBBW,BWGrass}, for more details.

It now follows from the Cuntz--Quillen theorem \cite[\textsection 8.3]{Varilly}  that any associated module is projective as a left $B$-module. In particular, since any quantum homogeneous space is a principal comodule algebra, every relative Hopf module is projective as a left $B$-module.


\bibliographystyle{siam}

\end{document}